\documentclass[a4paper,11pt,reqno]{amsproc} 
\usepackage{amssymb}
\usepackage{amsthm}
\usepackage{epsf} 
\usepackage{float}
\usepackage{bm}
\usepackage[normalem]{ulem}
\usepackage[hidelinks]{hyperref}
\usepackage[dvipsnames]{xcolor}
\usepackage[T1]{fontenc}   

\usepackage{tikz}     
\usetikzlibrary{arrows.meta}
\makeindex
\newtheorem{thm}{Theorem}[section]   
\newtheorem{cor}[thm]{Corollary}

\newtheorem{lemma}[thm]{Lemma}
\newtheorem{prop}[thm]{Proposition}
\newtheorem{example}[thm]{Example}
\newtheorem{defn}[thm]{Definition}

\newtheorem{rem}[thm]{Remark}

\newcommand{\Fqm}{\mathbb{F}_{q^m}}
\newcommand{\xx}{\mathbf{x}}

\def\min{\operatorname{min}}

\def\c1{\operatorname{c_1}}
\def\c2{\operatorname{c_2}}

\def\R{{\mathcal R}}

\def\E{{\mathcal E}}

\def\K{{\mathcal K}}
\def\M{{\mathcal M}}

\def\+{\oplus}                   
\def\*{\otimes}                  

\def\Fq{\mathbb{F}_{q}}

\def\Fq{\mathbb{F}_q}

\newcommand{\<}{\langle}
\renewcommand{\>}{\rangle}
\newcommand{\0}{\mathbf{0}}

\begin{document}
\title[The Euler characteristic, $q$-matroids, and a M{\"o}bius function]{The Euler characteristic, $q$-matroids, and a M{\"o}bius function}

\author[Johnsen]{Trygve Johnsen}
\address{Department of Mathematics and Statistics, 
 UiT-The Arctic University of Norway  
N-9037 Troms{\o}, Norway}
\email{trygve.johnsen@uit.no}

\author[Pratihar]{Rakhi Pratihar}

\address{Department of Mathematics and Statistics, 
 UiT-The Arctic University of Norway  
N-9037 Troms{\o}, Norway}
\email{pratihar.rakhi@gmail.com}

\author[Randrianarisoa]{Tovohery H. Randrianarisoa}
\address{Department of Mathematics,  
Florida Atlantic University,  
Boca Ratan, Florida, USA}
\email{tovo@aims.ac.za}
\thanks{ All authors have been partially supported by grant 280731 from the Research Council of Norway,  and  by the project ``Pure Mathematics in Norway" through the Trond Mohn Foundation and Troms{\o} Research Foundation.}

\date{\today}

\begin{abstract}We first give two new proofs of an old result that the reduced Euler characteristic of a matroid complex is equal to the M{\"o}bius number of the lattice of cycles of the matroid up to the sign. The purpose has been to find a model to establish an analogous result for the case of $q$-matroids and we find a relation between the Euler characteristic of the simplicial chain complex associated to a $q$-matroid complex and the lattice of $q$-cycles of the $q$-matroid. We use this formula to find the complete homology over $\mathbb{Z}$ of this shellable simplicial complex. We give a characterization of nonzero Euler characteristic for such order complexes. Finally, based on these results we remark why singular homology of a $q$-matroid equipped with order topology may not be effective to describe the $q$-cycles unlike the classical case of matroids.
\end{abstract}

\maketitle

\section{Introduction}
The Euler characteristic, a topological invariant, is a useful notion in algebraic topology and polyhedral combinatorics. For a simplicial complex $\Delta$ of dimension $k-1$, its (reduced) Euler characteristic is defined as $\chi(\Delta) = -f_0 +f_1 + \cdots + (-1)^{k-1} f_k$, where $f_i$ is the number of faces of $\Delta$ of cardinality $i$. Alternatively, for a simplicial complex, the (reduced) Euler characteristic is also defined as an alternating sum of the rank of its (reduced) singular homology groups, which are the same as the corresponding  simplicial homology groups and these ranks are called (reduced) Betti numbers of the simplicial complex. These two definitions of Euler characteristic indeed coincide, see, e.g., \cite[Theorem 1.7, Chapter 12]{GSSV}. Now for a special class of simplicial complexes, i.e., the shellable simplicial complexes, it is well known that only the top homology group is non-trivial with rank, say, $h_{k-1}$ (cf. \cite[Proposition 7.7.2]{B}) and thus $|\chi(\Delta)| = h_{k-1}$. Furthermore, for a matroid of rank $r$, the corresponding matroid complex $S_M$ (see Definition \ref{Defn. 2.17}), which is a shellable simplical complex of dimension $r-1$, has the Euler characteristic 
\begin{equation} \label{foundation}
\chi(S_M)=(-1)^{r-1}|\overline{\mu}(M^*)|.
\end {equation}
where $M^*$ is the dual matroid of $M$ and $$\overline{\mu}(M^*)= \begin{cases}|\mu(0,1)| &\text{ if } M^* \text{ has no loop},\\ 0 & \text{ otherwise}.
 \end{cases}
 $$
 with $\mu$ being the Möbius function on the lattice of 
flats of $M^*$ (see e.g. \cite[Proposition 7.4.7]{B}). 
The relation in Eq. (\ref{foundation}) can be seen as a link between the poset (lattice) of independent sets of $M$, and the geometric lattice of flats of $M^*$ and this has a very interesting application to coding theory. The relation is interesting in light of  \cite[Theorem 1]{JV13}, where one showed how generalized Hamming weights of linear codes, more generally, of matroids, can be determined via Betti numbers of Stanley-Reisner ring of an associated matroid complex. These Betti numbers, again, are determined by the
$\chi(S_{M_{|X})}$ for various subsets (cycles) $X$ of the ground set of this matroid $M$.

We are interested in a $q$-analogue of this relation wherein the notions of matroids and simplicial complexes are replaced by $q$-matroids and $q$-simplicial complexes, respectively. The notion of a $q$-matroid was introduced in \cite{JP} (see also Crapo \cite{Cr} where more general notion is studied). Jurrius and Pellikaan \cite{JP} also outline the connections between rank metric codes, which has been a topic of much current interest (see, e.g., the recent survery by Gorla \cite{G}), and $q$-matroids. A $q$-analogue of the notion of a simplicial complex goes back at least to Rota \cite{Rota} and it is also studied for more general notion of $q$-posets by 
Alder \cite{Ald} in his thesis. The combinatorics and the topology of the $q$-matroid complexes, i.e. $q$-simplicial complexes of independent spaces of $q$-matroids, have been studied in more recent papers \cite{GPR21, JPV21, P22}. The singular homology groups of a (punctured) $q$-matroid complex $\Delta_{\M}$ equipped with the order topology is determined partially in \cite{GPR21} and later completely determined in \cite{P22} by considering the order complex $I_{\M}$ corresponding to the $q$-matroid. The weak homotopy equivalence between the a punctured $q$-matroid complex $\mathring{\Delta}_{\M}$ and the associated order complex $I_{\M}$ implies that the absolute value of the Euler characteristic of the $I_{\M}$ is equal to the rank of the only nontrivial singular homology group of shellable $q$-complex. On the other hand, in \cite{JPV21} it has been proved that the generalized rank weights of a Gabidulin rank metric code can be determined by the Möbius numbers of the lattice of flats of an associated $q$-matroid (which is dual to the lattice of $q$-cycles of the dual $q$-matroid.) Thus it is relevant to pose the question, whether a relation between $\chi(I_{\M})$ and the Möbius numbers of the lattice of cycles (or, dually, the lattice of flats) of the $q$-complexes might help to determine the generalized rank weights of a Gabidulin rank metric codes by the singular homology groups. This is a question unanswered in \cite{GPR21}.

Our main result is Theorem \ref{conj:1} where we give a result, analogous to (\ref{foundation}), but different from it. It recovers (\ref{foundation}) if we set $q=1$. As far as we can see, however, it unfortunately does not enable us to read $\mu(0,1)$ off, directly from $\chi(I_{\M})$ and the rank of the mentioned homology group. But we believe that it can be an important theoretical result, and it also provides a link between data of the lattice of flats of $\M^*$ and the lattice of independent sets of $\M$. Other than that, we obtain a characterization of nonzero $\chi$-value of the order complex $\chi(I_{\M}$) with respect to some intrinsic properties of the ground space of the $q$-matroid. It turns out that the $q$-analogue of the result "$\chi(I_M) = 0$ is equivalent to $M^*$ has a loop", is not true. (For usual matroids $\chi(I_M)=\chi(S_M)$, also, as we will see in Section \ref{prelim},  where $ S_M$ is the simplicial complex of independent sets, and $I_M$ is the corresponding order complex). This  equivalence played a crucial role in \cite{JV13} in determining the generalized Hamming weights from the graded Betti numbers of the  Stanley-Reisner ring obtained from  matroid of the parity check matrices for the linear code. This indicates that there is no easy way to determine the generalized rank weights,  by considering the order topology on the poset of the associated $q$-matroid complex of the code.

The paper is organized as follows. We collect some preliminaries about the basic notions such as $q$-matroids, $q$-complexes, vector rank metric codes, finite topological spaces in Section \ref{prelim}. Two new proofs of the known result on the relation between Euler characteristic and M{\"o}bius numbers of a matroid complex are provided in Section \ref{classical}. These proofs are perhaps more straightforward than the one proved in \cite{B} using the Tutte polynomial and broken circuits. These proofs work as model for the proof in the $q$-analogue case that we prove in Section \ref{qproof}. In Section \ref{sec:5}, we show that for nontrivial $q$-matroids, the Euler characteristic of the corresponding order complexes is always nonzero, unlike in the case of classical matroids. We summarize in Section \ref{sec:6} the results and their implications in answering whether the generalized rank weights of a $q$-matroid can be determined by the singular homology groups of the corresponding $q$-complex.

\textbf{Notations} 
$E$ - $\{1, \ldots, n\}$

$\E$ - $\Fq^n$

$S$ or $S_M$ - Simplicial complex corresponding to a matroid $M$

$\chi(S)$ - The reduced Euler characteristic of a simplicial complex $S$

$\Delta$ or $\Delta_{\M}$ - (not simplicial, but only) $q$-complex corresponding to a $q$-matroid $\M$.

$I_{M}$ - order complex associated to a matroid $M$

$I_{\M}$ - order complex associated to a $q$-matroid $\M$

\section{Preliminaries} \label{prelim}

Throughout this paper, $q$ denotes a prime power and $\Fq$ the finite field with $q$ elements. Fixing a positive integer $n$, we use $E$ to denote a finite set with $n$ elements and $\mathcal{E}$ to denote an $n$-dimensional vector space over $\Fq$. For a non-negative integer $k \leq n$, we denote by ${n \brack k}_q$ the $q$-binomial coefficient given by $\frac{\prod_{i=0}^{k-1}(q^n - q^i)}{\prod_{i=0}^{k-1}(q^k-q^i)}$. The set of all subsets of $E$ is denoted by $2^E$, whereas $\Sigma(\mathcal{E})$ denotes the set of all $\Fq$-subspaces of $\mathcal{E}$. For $U_1, \ldots, U_r \in \Sigma(\mathcal{E})$, $\langle U_1, \ldots, U_r\rangle$ denotes the $\Fq$-space generated by $U_i$'s. For any $U \in \Sigma(\mathcal{E})$, $\dim U$ is the $\Fq$-dimension of $U$, if not otherwise stated.

We recall basic definitions and results concerning the notions such as ($q$-)matroids, ($q$-)shellability, ($q$-)simplicial complexes, and finite topological spaces etc. We mention relevant references for each of these topics for details.

\subsection{($q$-)Matroids}
   

\begin{defn} \label{def:matroid} 
A matroid is a pair $(E,r)$ where $E$ is a finite set and $r: 2^E \rightarrow \mathbb{N}_0$ is a function satisfying: \begin{list}{}{\leftmargin=1em\topsep=1.5mm\itemsep=1mm}
\item[{\rm (R1)}] If $X \subseteq E$, then $0 \leqslant r(X) \leqslant|X|$,
\item[{\rm (R2)}] If $X \subset Y \subset E$, then $r(X) \le r(Y)$,
\item[{\rm (R3)}] If $X,Y$ are subsets of $E$, then \[r(X \cap Y) + r(X \cup Y) \leqslant r(X)+r(Y).\]
\end{list}
\end{defn}

The function $r$ is called the rank function of the matroid. The rank of a matroid $M=(E,r)$ is $r(E)$. The nullity function $n$ of the matroid is given by $n(X) = |X|-r(X)$ for $X \subset E$. By (R1), this is an integer-valued non-negative function on $2^E$.
\begin{defn}
An independent set of a matroid is an $X$ such that $r(X)=|X|.$ Subsets of independent sets are then always independent.
\end{defn}

\begin{defn}
Let $\E=\mathbb{F}_q^n$, and let $\Sigma(\E)$ be the set of all $\mathbb{F}_q$-spaces of $\E$. A $q$-{\em matroid} is an ordered pair $\M= (\E,\rho)$  
where the function $\rho:\Sigma(\E)\to\mathbb{N}_0$
satisfies (P1)--(P3) below: 
\begin{list}{}{\leftmargin=1em\topsep=1.5mm\itemsep=1mm}
\item[{\rm (P1)}]  $0\leq \rho(X) \leq \dim X$ for all $X \in \Sigma(\E)$;
\item[{\rm (P2)}] $\rho(X) \le \rho(Y)$ for all $X,Y \in \Sigma(\E)$ with $X \subseteq Y$;
\item[{\rm (P3)}] $\rho(X+Y)+\rho(X \cap Y) \le \rho(X)+\rho(Y)$, for all $X,Y \in \Sigma(E)$.
\end{list}
\end{defn}

 The function $\rho$ is called the rank function of the $q$-matroid. The rank of a $q$-matroid $\M=(\E,\rho)$ is $rank(\M) = \rho(E)$. The nullity function $\eta$ of the $q$-matroid is given by $\eta(X) = \dim_{\Fq} X-\rho(X)$ for $X \in \Sigma(E)$. By (P1), this is an integer-valued non-negative function on $\Sigma(E)$.

For a $q$-matroid $\M=(\E,\rho)$, let $N_i=\{ X\in \Sigma(\E)\colon \eta(X)=i\}$. The minimal (w.r.t the inclusion) elements  of $N_i$ are called a \emph{$q$-cycles} of nullity $i$ of $M$. The $q$-cycles of nullity $1$ are the $q$-circuits of the $q$-matroid. 
\begin{defn}
Let $\mathcal{M}=(\E,\rho)$ be a $q$-matroid. Then a subspace $F \subseteq \E$ is called a \textit{$q$-flat} if $\rho(F \oplus \langle e \rangle) > \rho(F)$ for all $e \in \E \backslash F$.
\end{defn}

The following is well known. see e.g. \cite[Lemma 13]{JPV21}:
\begin{lemma} \label{fundamental}
Let $\mathcal{M} = (\E,\rho)$ be a $q$-matroid. Then $X \in \Sigma(\E)$ is a $q$-flat (of rank $r$) of a $q$-matroid $\mathcal{N}$ if and only if its orthogonal complement $X^{\perp}$ is a $q$-cycle (of nullity $\rho(\mathcal{M})-r$) for $\mathcal{M}^*$.
\end{lemma}

We also have:
\begin{defn}
And independent element of $\Sigma(\E)$ is a subspace $U$ of $\E$ such that $\rho(U)= \dim U$
(subspaces of independent spaces are then always independent).
A basis of $\E$ is a subspace $U$ of $\E$, which is not strictly contained in any independent subspace.
Let $\mathcal{B}$ denote the set of bases.
\end{defn}

The set of bases  $\mathcal{B}$ satisfies:
\begin{itemize}
\item[(B1)] $\mathcal{B} \ne \emptyset.$
\item[(B2)] For all $B_1, B_2 \in  \mathcal{B}$, if $B_1 \subset B_2$, then $B_1 = B_2.$
\item[(B3)] For all $B_1, B_2 \in  \mathcal{B}$ and for every subspace $A$ of codimension $1$ in $B_1$ satisfying
$B_1 \cap B_2 \subset A$, there is a $1$-dimensional subspace $u$ of $B_2$ such that $A + u \in B$.
\item[(B4)] For all $A, B \in \Sigma(E)$, if $I$ and $J$ are maximal intersections of some members of
$\mathcal{B}$ with $A$ and $B$, respectively, there exists a maximal intersection of a basis and
$A + B$ that is contained in $I + J$.

\end{itemize}
\begin{rem}
{\rm It is well known (See \cite[Theorem 6.1]{JP})  that if one takes (B1),(B2),(B3),(B4) as axioms and defines independent spaces $V$ as spaces contained in some basis, and defines $\rho(U)$, for any $U \in \Sigma(E)$, as the largest $\dim V$, for the and independent $V$ contained in $U$, then $\rho$ satisfies (P1),(P2),(P3). Hence (P1),(P2),(P3) and (B1),(B2),(B3),(B4)   are equivalent axiom systems. The fact that one needs an "extra" axiom (B4) reveals a profound difference between $q$-matroids and usual matroids. We will use this later, in Lemma \ref{nocommonx}, which has no analogue for usual matroids.}
\end{rem}
Since matroids (and $q$-matroids) give rise to  posets, in fact lattices, of ($q$)-flats and ($q$)-cycles ordered by inclusion, the following definition of  M{\"o}bius numbers, taken from \cite[p. 344]{R}, will also be useful: 
\begin{defn}\cite{R}
For a locally finite poset $(L,<)$ we set 
$\mu(x,x)=1$ for all nodes $x$, and if $x \le y$ for a node $y$ we set $\mu(x,y) = -\Sigma_{x \le z<y} \mu(x,z)$.
Moreover we set $\mu(x)=\mu(0,x)$ and $\mu(L)=\mu(0,1)$ if $L$ has a minimum element $0$ and a maximum element $1$.
\end{defn}

The following result will also be useful:
\begin{prop}\cite{Stanley}\label{q-binom}
The $q$-analogue of the binomial theorem is 
\begin{equation}\label{eq:q-binomial}
\prod\limits_{k=0}^{n-1} (1 - q^kt) = \sum\limits_{k=0}^{n} (-1)^kq^{{k \choose 2}}{n \brack k}_q t^k.
\end{equation}
\end{prop} 

The Equation \eqref{eq:q-binomial} is known as \emph{$q$-binomial theorem} and for a proof, see, e.g., \cite[p. 75]{Stanley}.

\subsection{Shellability}\label{2.2}

We recall the relevant definitions and results regarding shellable simplicial complexes and shellable $q$-simplicial complexes (or $q$-complexes, in short). For more on simplicial complexes one can refer to \cite{Stan}. 
\begin{defn}
A simplicial complex $S$ is a collection $S$ of subsets of a finite set $V$ such that if $B \in
S$, and $A \subset B$, then $A \in S$.

The elements in $S$ are called 
faces. A face that is not properly contained in any other face is called a
facet. 
\end{defn}

\begin{defn}\cite{B}
Let $S$ be a pure simplicial complex. A shelling of $S$ is a linear order of the
facets of $S$ such that each facet meets the complex generated by its predecessors
in a non-void union of maximal proper faces.That is: The linear order
$F_1,\cdots,F_s$, of the facets of $S$ is a shelling if and only if
for each pair $F_i, F_j$ of facets such that $1 < i < j < t$ there is a facet $F_k$ satisfying
$1<k<j$ and an element $x \in F_j$ such that $F_i \cap F_j \subset F_k \cap F_j= F_j- x$. 
A complex is said to be shellable if it is pure and admits a shelling. \end{defn}

For a shellable simplicial complex, the \emph{restriction operator} is defined as follows.

\begin{defn}\cite{B}
Let $S$ be a pure simplicial complex with a shelling $F_1, \ldots, F_t$ on its facets. Let $S_i$ be the subcomplex generated by the first $i$ facets of $S$. Then
\[
\mathcal{R}(F_i) := \{x \in F_i \colon F_i \backslash x \in S_i\}.
\]
\end{defn}

An important class of simplicial complexes that is known to be shellable is of matroid complexes, i.e., complexes formed by the independent sets of matroids. For relevant background on shellability for simplicial complexes and proof of this result we refer to a survey article \cite{B} by Bj{\"o}rner. Now we recall $q$-analogues of these notions.

\begin{defn}\cite[Section 5]{Rota} \cite[Definition 2.1]{GPR21} By a $q$-complex on $E=\mathbb{F}_q^n$,  we mean a subset $\Delta$ of $\Sigma(E)$ satisfying
the property that for every $A \in \Delta$, all subspaces of $A$ are in $\Delta$.
Let $\Delta$ be a $q$-complex. Elements of $\Delta$ are called faces of $\Delta$. Faces of $\Delta$ that
are maximal (w.r.t. inclusion) are called the facets of $\Delta$. The dimension of $\Delta$ is
max $\{\dim A | A \in \Delta\}$, and it is denoted by $\dim \Delta$. We say that $\Delta$ is pure if all its facets have the same dimension.
\end{defn}
For a set $\{F_i\}_i$ of facets of $\Delta$, let 
$<\{F_i\}_i>$ be the subcomplex consisting of all subspaces of at least one of the $F_i.$
\begin{defn}\cite[Definition 4.1]{GPR21} \cite{Ald}
Let $\Delta$ be a pure $q$-complex on $\Sigma(E)$.
 A shelling of $\Delta$ is a linear
order $F_1,\cdots,F_t$ on the facets of $\Delta$ such that for each $j = 2,\cdots,t$, the $q$-complex $<F_j> \cap <F_1,\ldots,F_{j-1}>$ is generated by a nonempty set of maximal proper faces of $F_j$.
We say that a $q$-complex is $q$-shellable if it is pure and it admits a shelling.
\end{defn}
\begin{rem}
In \cite{Ald} and \cite{GPR21}, $q$-shellable $q$-complex were simply called shellable. Since, we are dealing with multiple objects, i.e., simplicial complexes and $q$-complexes, it is better to give a more precise notion of shellability.
\end{rem}

\begin{rem}
By abuse of notation, if $\prec$ is the ordering on the facets of $\Delta$ which defines the shelling, we simply say that $\prec$ is the shelling of $\Delta$.
\end{rem}

We recall that for matroids, subsets of independent sets are always independent, and for 
$q$-matroids subspaces of independent spaces are always independent. Hence the sets of independent sets and independent spaces form simplicial complexes and $q$-complexes, respectively.

\begin{defn}\label{Defn. 2.17}\ 
\begin{itemize}
\item For a matroid $M$, let $S_{M}$ be the simplicial complex consisting of its independent sets. $S_{M}$ is called the matroid complex associated to $M$.
\item For a $q$-matroid $\M$, let $\Delta_{\M}$ be the $q$-complex formed by the independent spaces of $\M$. $\Delta_{\M}$ is called the $q$-matroid complex associated to $\M$.
\end{itemize}
\end{defn}

Let $\E=\Fq^n$ and let $\M=(\E,\rho)$ be a $q$-matroid and let $\Delta_\M$ be the associated $q$-matroid complexes. Fix an ordering $0\prec 1\prec \alpha_2 \prec \cdots \prec \alpha_q$ on the elements of $\Fq$. It is important that $0$ is the smallest and $1$ is the next smallest element. This induces a lexicographic ordering, which we also denote by $\prec$, on the elements of $\Fq^n$. This further induces a lexicographic ordering, denoted again by $\prec$ on $(\Fq^n)^k$, where 
\[ 
(\Fq^n)^k = \{(v_1, \ldots, v_k) \colon v_i \in \Fq^n \}.
\]

For $k\leq n$, let $G(k,n)$ denote the Grassmannian of the $k$-spaces on $\Fq^n$. Each element $U$ of $G(k,n)$ corresponds to a unique generator matrix $G_U=[u_k,u_{k-1},\ldots, u_1]^T$ of order $k\times n$ and rank $k$ in reduced row echelon form, where $u_i$ is a column vector in $\Fq^n$. Using this correspondence, each element $U$ of $G(k,n)$ is therefore associated to a unique $(u_1,\dots,u_k)\in (\Fq^n)^k$ which we call the \emph{reduced generator} of $U$. We now define an ordering on the Grassmannian $G(k,n)$.

\begin{defn}\label{prec}
Let $U$ and $V$ be two elements of $G(k,n)$ such that $(u_1,\dots,u_k)\in (\Fq^n)^k$ and $(v_1,\dots,v_k)\in (\Fq^n)^k$ are the respective reduced generators. We define an ordering $\prec_q^k$ on $G(k,n)$ as the lexicographic ordering on the reduced generators, i.e.,
\[
U\prec_q^k V \Longleftrightarrow (u_1,\dots,u_k)\prec (v_1,\dots,v_k)\in (\Fq^n)^k.
\]
\end{defn}

\begin{thm}\cite[Theorem 4.4]{GPR21}\label{q-shell}
Let $\Delta_\M$ be a $q$-matroid complex of dimension $r$. Then $\prec_q^r$ is a shelling on $\Delta_\M$. In other words, $\Delta_\M$ is a shellable $q$-complex.
\end{thm}

We now give another example of shellable simplicial complexes which are the key objects of this paper. It has been shown in \cite{P22} that these simplicial complexes are not matroid complexes.

\begin{defn}[Order complex] 
Let $\Delta$ be a $q$-complex. The order complex $I$ associated to $\Delta$ is the set of all finite chains of nonzero elements of $\Delta$, i.e.,
\[
I = \{ (U_1\subset \cdots \subset U_t)\colon U_i\in \Delta \}.
\]
The order complex is a simplicial complex and if $\Delta_\M$ is a $q$-matroid complex associated to a matroid $\M$, the associated order complex is denoted by $I_\M$.
\end{defn}

A chain $U=(U_0\subset \cdots \subset U_k)$ is called unrefinable if $\dim U_i = i$ for all $0\leq i\leq k$. Given two unrefinable chains $U=(U_0\subset \cdots \subset U_k)$ and $V=(V_0\subset \cdots \subset V_k)$, we define the reverse lexicographic ordering induced by the ordering $\prec_q^i$ on the Grassmannians $G(i,n)$ for $1\leq i\leq k$:
$U \prec_\ell V$ if and only if $e$ is the largest index such that $U_e\neq V_e$ and $U_e\prec_q^e V_e$.

\begin{thm}\cite{P22}
Let $\M$ be a $q$-matroid of rank $r$ and let $\Delta_\M$ be the associated $q$-matroid complex. Furthermore, let $I_\M$ be the order complex associated to $\M$. Then the reverse lexicographic ordering $\prec_\ell$ induced by $\prec_q^l$ is a shelling on $I_\M$. In other words, the order complex $I_\M$ of a $q$-matroid complex $\M$ is a shellable simplicial complex.
\end{thm}

\subsection{Topological preliminaries}

Finite topological spaces are simply topological spaces having only a finite number of points. We simply call these finite spaces. The study of finite spaces was introduced by Alexandroff \cite{Alex} and thereafter, continued by Stong \cite{Stong} and McCord \cite{McCord}. For the theory of finite spaces, we refer to the expository article by May \cite{May2012} and the book by Barmak \cite{Bar11}. In this subsection, we discuss finite spaces induced from certain finite posets related to matroids and $q$-matroids. Also, for the convenience of the reader, we provide a quick review of the relevant notions and results.

First we recall the definitions of two topological spaces associated to a finite simplicial complex and a finite poset, respectively.

\begin{defn}
{\rm Let $S$ be a finite simplicial complex on a vertex set $V$. Let $|S|$ denote the subset of all functions $f : V \rightarrow [0,1]$ such that

\begin{enumerate}
    \item supp$(f) := \{v \in V \colon f(v) \neq 0\}$ is a simplex in $S$,
    
    \item $\sum_{v \in S} f(v)=1$.
\end{enumerate}
 Then taking the usual Euclidean topology on $[0,1]$, the set $|S|$ equipped with the subspace topology of the product space $[0,1]^{\#V}$ is called the \emph{geometric realization} of the simplicial complex $S$.}
\end{defn}
 Thus, in short, the geometric realization $|S|$ of a finite simplicial complex $S$ is a compact subspace of some real Euclidean space. Next we define a finite space which shows an interesting connection between finite posets and finite $T_0$ spaces (but not $T_1$)\footnote{Recall that a topological space $X$ is $T_0$ if given any two distinct points of $X$, at least one of them is contained in an open set that does not contain the other point.}. Note that, if the finite spaces are $T_1$, the topology is necessarily discrete since any two points are separated by two open sets which contain only one of them. Thus the finite $T_1$ spaces are not so interesting, whereas the finite $T_0$ spaces have a rich structure.
 
\begin{defn}\label{ordertopo}
{\rm Let $(P, \leq)$ be a finite poset. A subset $U \subseteq P$ is called a \emph{down-set} if for any $y \in U$ and $x \in P$ with $x \leq y$, we have $x \in U$. The \emph{order topology} $\mathcal{K}(P)$ on $P$ is defined by declaring that the open sets in $\mathcal{K}(P)$ are precisely the down-sets in $P$. } 
\end{defn}
The order topology makes the finite poset a $T_0$ space. To see the reverse correspondence, i.e. from a finite $T_0$ space to a finite poset, we refer to \cite{Bar11} or \cite[Section 5.1]{GPR21}.

Let $P$ be either the poset of non-empty simplices of a finite simplicial complex $S$, or a poset of non-zero subspaces of a finite-dimensional vector space of a finite field, ordered by inclusion. We can associate three different kinds of topological spaces:

\begin{enumerate}
\item $\mathcal{K}(P)$ is the order topology of Definition \ref{ordertopo}.

\item Following for example \cite[Section 5]{McCord} $P$  determines a (new) simplicial complex $I(P)$: Any $i$-simplex of $I(P)$ is a chain $A_0 \subseteq A_1 \subseteq \cdots \subseteq A_i$ of subsets of $P$ of length $i$. This simplicial complex whose faces are chains in $P$ is called the \emph{order complex} of the poset $P$. Now one can consider the Euclidean space $|I(P)|$ which is the geometric realization of the finite simplicial complex $I(P)$.

\item Since also the simplicial complex from (2) has an underlying poset structure, one can consider the order topology $\mathcal{K}(I(P))$ of the chain complex $I(P)$.
\end{enumerate}

Additionally,
\begin{enumerate}
\item[(4)] if the original poset $P$ is a simplicial complex $S$ then we have one more topological space, which is the geometric realization $|S|$ of $S$.
\end{enumerate}

Since the posets corresponding to both matroid and $q$-matroid complexes have a unique minimal element, which are the empty set and the zero space, respectively, these posets are contractible. Thus the singular homology groups are trivial. Therefore, we will consider the posets with the unique minimal element removed. 

\textbf{Notation}: For the rest of this paper, we use $\mathring{S}$ and $\mathring{\Delta}$ to denote a punctured simplicial complex $S \backslash \{\emptyset\}$ and a punctured $q$-simplicial complex $\Delta \backslash \{\{0\}\}$, respectively. The notation ${I}_M$ is used to denote the order complex corresponding to the punctured simplicial complex $\mathring{S}_M$ associated to the matroid $M$, i.e., $I_M=I(\mathring{S}_M)$. We use the same notations to denote the underlying posets too.
\begin{thm}\cite[Theorem 1.4.6]{Bar11}\label{whe1}
Let $X$ be a finite $T_0$ space and $I(X)$ be the order (or the chain) complex of the poset associated to $X$. Then $X$ and $|I(X)|$ are weak homotopy equivalent.
\end{thm}

\begin{thm}\cite[Theorem 1.4.12]{Bar11}\label{whe2}
Let $S$ be a finite simplicial complex. Then its geometric realization $|S|$ is weak homotopy equivalent to $\mathcal{K}(\mathring{S})$ which is equipped with the order topology.
\end{thm}

As a consequence of the above theorems, we observe the following key points that work as a main motivation behind our work.

\begin{prop} \label{equival}
Let $M$ be a matroid and $S_M$ be the matroid complex and $I_M$ be the order complex corresponding to $M$. Then the four topological spaces $|{S}_M|$, $\mathcal{K}(\mathring{S}_M),$
$|{I}_M|$, $\mathcal{K}(\mathring{I}_M)$ are weak homotopy equivalent
and have the same singular homology, which is also equal to the simplicial homology of $\mathring{S}_M$ and $\mathring{I}_M$.
\end{prop}
\begin{proof}
 The weak homotopy equivalence of $\mathcal{K}(\mathring{S}_M),$ and $|{I}_M|$ follows from Theorem \ref{whe1}. The weak homotopy equivalence between $|{I}_M|$ and $\mathcal{K}(\mathring{I}_M)$ and that between $|{S}_M|$ and $ \mathcal{K}(\mathring{S}_M)$ follows from Theorem \ref{whe2}.
\end{proof}
Using Theorem \ref{whe1} (or \cite[Corollory 1]{McCord}) the analogous results for a $q$-matroid complex can be summarized as follows. 

\begin{prop} \label{equival2}
Let $\mathcal{M}$ be a $q$-matroid and $\Delta_{\M}$ be the $q$-complex and $I_{\M}$ be the order complex corresponding to $\M$. Then the spaces $\mathcal{K}(\mathring{\Delta}_{\M}),$  $|{I}_{\M}|$ $ \mathcal{K}(\mathring{I}_{\M})$ are weak homotopy equivalent and have the same singular homology, which is also equal to the simplicial homology of ${I}_{\M}$. 
\end{prop}

\section{The Euler characteristic and the Möbius function} \label{classical}

In this section we prove  a well known result (Theorem \ref{main1}) for usual matroids, in two ways. The first proof is meant as a key to the understanding of the second proof. The result itself, and the second proof,  serve as a motivation for the main result of our article, which is Theorem \ref{conj:1}, and as a pattern for the proof of Theorem \ref{conj:1}, respectively.

We now quote a result by G. Rota (The original notation $q_i$ by Rota has been changed to $\lambda_i$ for each $i$) \cite[Proposition 1, p. 349]{R}:
 \begin{prop} \label{rota}
 Let $R$ be a subset of a finite lattice $L$ such that the maximum element $ 1 \textrm{ is not in }R$, and for every $x \in L,$ except $x=1,$ there is an element $y \in R$ such that $y \ge x$. For $k \ge 2$, let $\lambda_k$ be the number of subsets of $R$ containing $k$ elements whose meet is the minimum element $0$. Then $\mu(0,1)=\lambda_2 -\lambda_3+\cdots (-1)^s\lambda_s.$
 \end{prop}
 
 Now we state the straightforward corollary of the above result that we use in the sequel.
 \begin{cor}
 Let $M = (E, \rho)$ be a matroid of nullity at least $2$ and $L_M$ be the lattice of its cycles. Suppose $\{C_1, \ldots, C_s\}$ is the set of circuits of $M$ and we take $S = \{1, \ldots, s\}$. If $\lambda_i$ be the number of subsets $X \subseteq S$ of cardinality $i$ such that $\bigcup_{i \in X} C_i = E $, then $\mu_{L_M} = \lambda_2 - \lambda_3 + \cdots + (-1)^s \lambda_s.$   
 \end{cor}
 
 \begin{proof}
  We apply Proposition \ref{rota} to the lattice of flats of the dual matroid $M^*$ and take $R$ to be the set of hyperplanes of $M^*$ (which correspond to circuits of $M$). It is a well known result that the lattice of flats of $M^*$ is the dual of the lattice of cycles of $M$. Also, that the meet of some hyperplanes is zero implies that the join of the corresponding circuits is the ground set $E$. Moreover, $\mu(0,1)$ is the same for a lattice and its opposite lattice. Thus we the immediately obtain that
  $$\mu(0,1) = \lambda_2 -\lambda_3+\cdots (-1)^s\lambda_s$$ for the lattice $L_M$.
 \end{proof}

\subsection{The classical case} \label{class}

\begin{thm} \label{main1}
Let $M=(E,\rho)$ be a co-loopless matroid of rank $r$. Assume that $S_M$ is the simplicial complex of independent sets of $M$ and $L_M$  the lattice of cycles of $M$. Then the reduced Euler characteristic $\chi(S_{M})$ of $S_M$ is equal to $(-1)^{r-1}|\mu_{L_M}(0,1)|$ , where $\mu_{L_M}(0,1)$ is the M{\"o}bius number of the lattice $L_M$.
\end{thm}
\begin{rem}
For a different, classical, proof, see \cite[Prop. 7.4.7,(i)]{B}. This result was used in an essential way, both in \cite{JV21}, and in \cite{JPV21}.
\end{rem}
\begin{proof}

We first prove the equality of the absolute values of $\chi(S_{M})$ and $(-1)^{r-1}\mu_{L_M}(0,1)$. We will take care of the signs at the end.

Since $M$ has no co-loop, i.e., there is no element in $M$ which is contained in all the bases of $M$. Therefore, $M$ has nullity at least $1$. If the nullity is $1$, then the ground set $E$ is a circuit. Thus all the proper subsets of $E$ the are independent and hence the Euler characteristic will be $1$ or $-1$, while $\mu(0,1)$ obviously is $-1$. This proves the case of nullity $n(M)$ equal to $1$.

Now we consider the case when the nullity $n(M)$ is at least $2$. 
Let $d_k$ be the number of dependent sets of cardinality $k$, and $f_k$ the number of independent sets. The Euler characteristic is defined as
 \begin{equation}\label{Dexpr}
     \chi(S_{M})=\sum\limits_{k=0}^n (-1)^{k-1}f_k= \sum\limits_{k=0}^n (-1)^{k}d_k.
 \end{equation}

 Then $|\chi(S_{M})|=|\sum\limits_k(-1)^kd_k|$, and we have used that $\sum\limits_k(-1)^k(d_k+f_k)=0$ (since we may assume that the ground set of the matroid has at least $2$ elements).

Suppose $\{C_1,\cdots,C_s\}$ is the set of the circuits of $M$. Let $D_{k,i}$ be the set of (dependent) subsets of $E$
of cardinality $k$ containing $C_i$.
Clearly, $d_k=|\bigcup_{i=1}^{s} D_{k,i}|$. If we let $S=\{1,2,\cdots,s\}$, then by the principle of inclusion/exclusion, we get:
$$d_k=\sum\limits_{i=1}^s(-1)^{i-1}\sum_{\substack{X\subset S\\ |X|=i}}|D_{k,X}|, $$ 
where $D_{k,X}=\bigcap_{i \in X} D_{k,i}.$  But clearly 
$|D_{k,X}|$ is also the number of dependent sets of cardinality $k$ containing $C_X=\bigcup_{i \in X}C_i$.

For each $X \subset S$, there are now two cases:

\textbf{Case 1}: $C_X \ne E$ or equivalently, the join of the nodes of $L_M$ corresponding to the circuits in $X$ is not equal to $1$. 

In this case, the alternating sum $\sum\limits_{k=0}^{n}(-1)^{|E|-|C_X|}|D_{k,X}|$ is the same as the alternating sum of the number of sets containing $C_X$ and contained in $E$. Equivalently, this is the alternating sum (with a possible shift in sign) of the number of subsets of $E- C_X$. Since, $C_X \ne E$, considering the Pascal's triangle it is clear that $\sum\limits_{k=0}^{n}(-1)^{|E|-|C_X|}|D_{k,X}| = 0$.

\textbf{Case 2}: $C_X = E$ or equivalently, the join $\vee$ of the nodes of $L_M$ corresponding to the circuits in $X$ is equal to $1$ (corresponding to $E$).

In this case the "alternating sum" contains just one term
$(-1)^{|E|}.$ The contribution to $\Sigma_k(-1)^kd_k$
from the subsets of $S$ of cardinality $i$ is then
$(-1)^{|E|}(-1)^{i-1}$ for each of them, and the contribution for all of them taken together is then
$(-1)^{|E|+i-1}\lambda_i$, where $\lambda_i$ is the number of subsets $X$ of $S$ of cardinality $i$, such that the  $\vee$ of the corresponding atoms of $L$ is equal to $E$. Hence  

 $\Sigma_k(-1)^kd_k= (-1)^{|E|-1}(\lambda_2 -\lambda_3+\cdots (-1)^s\lambda_s)$, and 
 $$|\Sigma_k(-1)^kd_k|=|\lambda_2 -\lambda_3+\cdots (-1)^s\lambda_s|.$$
 
 We the immediately obtain:
  $$|\chi(S_{M})|=|\Sigma_k(-1)^kd_k|=|\lambda_2 -\lambda_3+\cdots (-1)^s\lambda_s|=|\mu(0,1)|,$$ for the lattice $L_M$.
  
 As for signs, we know that, viewed as an alternating sum of homology numbers, 
 the sign of $\chi(S_{M})$ is $(-1)^{r-1}$.
 Hence $\chi(S_{M})=(-1)^{r-1}|\mu(0,1)|.$
 \end{proof}

 \begin{rem}\label{rem:3.5}
 If $M$ has a coloop, then there is an element which is in all the bases of $M$. Therefore, the top cycle, which is the union of all the circuits, is not equal to $E$. Thus we see from the notation introduced above that all the $\lambda_i$
 are zero. So the argument above for $n(M) \ge 2$, gives that $\chi(S_{M})=0$. Then one easily sees that $\chi(S_{M})=0$ if $M$ has a coloop.
  Hence we recover the well-known result that:
 $\chi(S_{M})=   (-1)^{r(M)-1}\overline{\mu}(M)$, where 
 $$\overline{\mu}(M)= \begin{cases}|\mu(L)| &\text{ if } M \text{ has no co-loop},\\ 0 & \text{ otherwise}.
 \end{cases}
 $$
 \end{rem}

We will now give a second proof of Theorem \ref{main1}.
We will utilize Proposition \ref{equival} and compute $\chi(I_M)$ instead.
Hence we study the chain complex, which has as non-empty simplices all chains of $S_{M}$ except those that contain $\emptyset$. We allow ourselves to change notation, and we now let  $f_k$ denote the number of these chains of length $k$, for all $k$ in question. For notational simplicity we will assume that $n(M) \ge 2$. The case $n(E)=1$ can be treated in a similar way.

Furthermore we now let  $d_k$ denote the number of chains of subsets of $E$, of length $k$, which do not contain $\emptyset$, and where the top term corresponds to a dependent set for $M$. And we let $s_k$ denote the number of all chains of subsets of $E$, of length $k$, which do not contain $\emptyset$.
Clearly $s_k=d_k+f_k$, for all $k$ in question.

We now have: $\Sigma (-1)^k s_k=0$.
Hence $|\Sigma (-1)^k f_k|=|\Sigma (-1)^k d_k|$.
This is a consequence of the two first items of the following result:
\begin{lemma} \label{help1}
\begin{itemize}
\item[(1)] The contribution to $\Sigma (-1)^k s_k=0$
from the chains that contain $E$ is $(-1)^{|E|}.$
\item[(2)] The contribution to $\Sigma (-1)^k s_k=0$
from the chains that do not contain $E$ is $(-1)^{|E|-1}.$
\item[(3)] The alternating sum of the cardinalities of the sets of chains of subsets, of given length $k$, all containing $\emptyset$, but not $E$, is $(-1)^{|E|}.$
\item[(4)] The alternating sum of the cardinalities of the sets of chains of subsets of given length $k$, containing both $\emptyset$, and $E$, is $(-1)^{|E|-1}.$
\item[(5)] The alternating sum of the cardinalities of the sets of all chains of subsets of $E$ of given length $k$, is $0.$
\end{itemize}
 
\end{lemma}
\begin{proof}

First we show that (1) implies (2), and then we prove (1) afterwards.

One sees that the chains from cases (1) and (2) taken together describe the chain complex associated to the simplicial complex consisting of all subsimplices of a single simplex with $n$ points. The $\chi$-value of this simplex, which is the same as that of the chain complex consisting of the chains from (1) and (2) together, is zero, so the statements in (1) implies that of (2).

Statement (1) is proved by induction. For $|E|=1$, it is obvious. Moreover,  since we now have seen (in the preceeding paragraph) that the sum  $\Sigma(-1)^ks_k=0$ for the simplicial complex representing all chains not containing $\emptyset$ (but including the empty chain of such non-empty sets), it follows by induction that 
the contribution to $\Sigma (-1)^k s_k=0$
from the chains that contain $E$, being minus the sum of the contributions from the chains with top term $F$, for $F$ strictly contained in $E$, and the contribution from the empty chain, is an alternating sum of all terms in a line of Pascal's triangle, except 
an end term 1, which appears with interchanging sign
as $E$ increases.

Part (3) follows from Part(1) by symmetry.
Part (4) follows from Part (2): 
One passes from the chains in (2) to those in (4), without changing the oddity of their lengths, by simply adding $\emptyset$ and $E$ to each of them.
Part (5) then follows from Parts (1)-(4). 

(Part (5) can also be proved in the following way if $|E|=n$ is even: Then complete chains are of the type $\emptyset=V_1 \subset \cdots \subset V_{n+1}=E$, with  $n+1$ being odd. Passing from the set of all chains of length $k$ to all chains of length $n+1-k$ can then by done by taking complements within each such complete chain, and since the oddity of the length changes, it is clear that statement (5) holds).
\end{proof}

We will now count the $d_k$, or rather $\Sigma (-1)^k d_k$. Let the bad chains be the ones not containing $\emptyset$, and where the top term represents a dependent set for $M$. The number of bad chains of cardinality $k$ is then  precisely $d_k$.
Let the circuits of $M$ be $C_1,\cdots, C_s$ as before.
Now we let $d_{k,i}$ be the set of bad (strict) chains $A_1 \subset \cdots\subset A_k$ of length $k$, with $A_k$ containing $C_i$ as a (not necessarily strict) subset.

The principle of inclusion/exclusion now gives:
$$d_k=\Sigma_{i=1}^s(-1)^{i-1}\Sigma_{\{X\subset S\colon |X|=i\}}|D_{k,X}|, $$ 
where $D_{k,X}=\cap_{i \in X} D_{k,i}.$  But clearly 
$|D_{k,X}|$ is also the number of bad chains of cardinality $k$ containing the union $\cup_{i \in X} C_i$. As above we have two cases:

(1) $C_X=E$ and (2) $C_X\ne E$. 

Or equivalently: 

(1)The wedge $\vee$ of the nodes of $L_M$ corresponding to the circuits in $X$ is equal to $1$ (corresponding to $E$).
(2) This wedge is something different.

In Case (1), it follows from Lemma \ref{help1}, part (1), that $\Sigma_k(-1)^k|D_{k,X}|=(-1)^{|E|}.$

In Case (2) the contribution to $\Sigma (-1)^k d_k$ from 
$(-1)^{s-1}\Sigma_k(-1)^k|D_{k,X}|$, will then by the same logic, be a sum of $(-1)^{s-1}(-1)^{|A|}$, over each set $A$, such that $\cup_{i \in X} C_i \subset A \subset E.$ The absolute value of this sum is then  like summing a row of length at least $2$ in Pascal's triangle (each summand is +-the number of $A$ of cardinality $|E|-j$ with $\cup_{i \in X} C_i \subset A \subset E.)$ Hence the contributions to $\Sigma (-1)^k d_k$ from each $\Sigma_k(-1)^k|d_{k,X}|$ in Case (2)
is zero. Since the contributions in Case (1) are all
$(-1)^{|E|}$, we get $$|\chi(S_{M})|=|\lambda_2 -\lambda_3+\cdots (-1)^s\lambda_s|=|\mu(L)|.$$ The arguments concerning the signs are as before.

\section{The $q$-analogue case}\label{qproof}

In this section, we prove a $q$-analogue of the result of the previous section, i.e., finding an expression for the reduced Euler characteristic $\chi(I_{\M})$ of the simplicial complex of chains of a $q$-matroid $\M= (E = \Fq^n, \rho)$, excluding the independent space $\{0\}$. In this case, we find an expression in terms of the properties of the (opposite-)geometric lattice $L= L_{\M}$ of $q$-cycles of $\M$. Here $I_{\M}$ is analogous to the chain complex $C(S_{M})$ of the previous section. Following the result in the classical case as mentioned in Remark \ref{rem:3.5}, we define the following notation for the $q$-matroid $\M$ with $L$ as its lattice of $q$-cycles. 
 $$\overline{\mu}(\M):= \begin{cases}|\mu(L)| &\text{ if } \M \text{ has no co-loop},\\ 0 & \text{ otherwise}.
 \end{cases}
 $$
We also fix the following notions and notations for the rest of this section. By a chain of subspaces of length $k$, we mean a flag of subspaces of the form $\{0\} \neq V_1 \subsetneq V_2 \subsetneq \cdots \subsetneq V_k$. Thus in the chain complex $I_{\M}$, chains of length $k$ corresponds to those chains of subspaces of $\Fq^n$ of length $k$ in which $V_k$ is an independent space of $\M$. We call these chains of length $k$ of $I_{\M}$ "legitimate" and let $f_k$ be the number of legitimate chains of $I_{\M}$ of length $k$. By definition, the empty chain is only legitimate chain of length $0$, and thus $f_0 =1$. On the other hand, we let $d_k$ be the number of "illegitimate" chains of length $k$ of $I_{\M}$, i.e., chains of the form $\{0\} \neq V_1 \subsetneq V_2 \subsetneq \cdots \subsetneq V_k$, where $V_k$ is dependent. We also let $s_k$ be the number of chains of length $k$ in $E=\Fq^n$ with $V_1 \neq 0$. If $k=0$, then $s_k = 1$. 

To determine $\chi(I_{\M}) = \sum\limits_{k=0}^{r} (-1)^{k-1}f_k$, we first prove the following lemma on determining some alternating sums of number of chains in $I_{\M}$ that will be useful to prove our main result in this section.

\begin{lemma} \label{help2}
Let $ E = \Fq^n$ and $s_k$ be the number of chains of subspaces of $E$ of length $k$, i.e., the chains of the form $\{0\} \neq V_1 \subsetneq V_2 \cdots \subsetneq V_k$. Then
\begin{enumerate}
\item  The contribution to $\sum\limits_{k=0}^{n} (-1)^k s_k$
from the chains that contain $E$ is $(-1)^{n}q^{{ n \choose\ 2}}.$ 
\item  The contribution to $\Sigma (-1)^k s_k$
from the chains that do not contain $E$ is $(-1)^{n-1}q^{{ n \choose\ 2}}.$ 
\item  The alternating sum of the cardinalities of the sets of chains of subspaces, of given length $k$, all containing $\{0\}$, but not $E$, is $(-1)^{n}q^{{ n \choose\ 2}}.$
\item  The alternating sum of the cardinalities of the sets of chains of subspaces of given length $k$, containing both $\{0\}$, and $E$, is $(-1)^{n-1}q^{{ n\choose\ 2}}.$
\item  The alternating sum of the cardinalities of the sets of all chains of subspaces of $E$ of given length $k$,  is $0.$
\end{enumerate}
 \end{lemma}
\begin{proof}

First we show that (1) implies (2), and then we prove (1) afterwards.

One sees that the chains from cases (1) and (2) taken together describe the chain complex associated to the $q$-complex consisting of all subspaces of a single vector space of dimension $n$. The $\chi$-value of this simplex, which is the same as that of the (contractible) chain complex consisting of the chains from (1) and (2) together, is zero, for example by \cite[Prop. 5.2 and Theorem 5.11]{GPR21}, so the statement in (1) implies that of (2).

Statement (1) is proved by induction. For $\dim E=1$, i.e., $n=1$, it is obvious that the value is $-1=(-1)^{1}q^{{ 1 \choose\ 2}}$. Moreover,  since we now have seen (in the preceeding paragraph) that the sum  $\sum\limits_{k=0}^{n}(-1)^k s_k=0$ for the simplicial complex representing chains not containing $\{0\}$ (but including the empty chain of such non-zero spaces), it follows by induction that 
the contribution to $\Sigma (-1)^k s_k=0$
from the chains that contain $E$, being minus the sum of the contributions from the chains with top term $F$, for $F$ strictly contained in $E$, and the contribution from the empty chain, is an alternating sum of type
\begin{equation}
   \Sigma_{i=0}^{n-1}(-1)^{i-1}q^{{i \choose\ 2}}{n \brack\ i}_q, 
\end{equation}
 using the induction hypothesis. Here ${ n \brack\ i}_q=|G(i,n)|$ 
for the Grassmannian $G(i,n)$ of $i$-spaces in $n$-space.
It follows from Proposition \ref{q-binom}, by substituting $t=1$,  that this alternating sum is $(-1)^nq^{{n \choose 2}}.$

Part (3) follows from Part(1) by symmetry.
Part (4) follows from Part (2): 
One passes from the chains in (2) to those in (4), without changing the oddity of their lengths, by simply adding $\{0\}$ and $E$ to each of them.
Part (5) then follows from Parts (1)-(4). 

(Part (5) can also be proved in the following way if $\dim \mathcal{E}=n$ is even: Then complete chains are of the type $\{0\}=V_1 \subset \cdots V_{n+1}=\mathcal{E}$, with  $n+1$ being odd. Passing from the set of all chains of length $k$ to all chains of length $n+1-k$ can then be done by taking complements within each such complete chain, and since the oddity of the length changes, it is clear that statement (5) holds).
\end{proof}

\begin{thm}\label{conj:1}
For a $q$-matroid $\M = (\mathcal{E}=\Fq^n, \rho)$, let $I_{\M}$ be the corresponding chain complex of non-zero independent spaces of $\M$.  
Let $\{C_i| i \in S =\{1,2,\cdots,s\}\}$ be the $q$-circuits of $\M$ and $C_X$ denotes the linear span of the circuits $C_i$ indexed by $X$. We use $\lambda_{i,{l}}$ to denote the number of subsets $X \subset S$, of cardinality $i$ such that the linear span $C_X$ of the circuits $C_i$ indexed by $X$, has codimension $l$ in $E$. Then we have
\begin{equation} \label{mainform}
\chi(I_{\M})= (-1)^{r-1}q^{{n \choose\ 2}}\overline{\mu}(\M)+\sum\limits_{l=1}^{n-1}\sum\limits_{i=1}^s\lambda_{i,l}\sum\limits_{j=0}^l(-1)^{n+i+j-1}q^{{n-j \choose\ 2}}{l \brack\ j}_q,
\end{equation}
 where ${l \brack\ j}_q$ is the cardinality of the Grassmannian $G(j,l)$ over $\mathbb{F}_q$.
\end{thm}

\begin{proof}

From Lemma \ref{help2}, part (5), we have that 
\[\chi(I_{\M}) = \sum\limits_{k=0}^{n} (-1)^{k-1}f_k = \sum\limits_{k=0}^{n} (-1)^k d_k.\]
 Now let $d_{k,i}$ be the set of illegitimate chains of length $k$ with $V_k$ containing the circuit $C_i$. If for $X \subset S$, we let $d_{k,X} = \bigcap\limits_{i\in X} d_{k,i}$, then by the principle of inclusion-exclusion, we have
 \begin{equation} \label{Dexpl}
d_k=\sum\limits_{i=1}^s(-1)^{i-1}\sum\limits_{\substack{X\subseteq S\\ |X|=i}}|d_{k,X}|. 
\end{equation}
Therefore,
\[\chi(I_{\M})=\sum\limits_{k=0}^{n} \sum\limits_{i=1}^{s} (-1)^k (-1)^{i-1} \sum\limits_{\substack{ X \subseteq S \\ |X|=i}} |d_{k,X}|.\]

But clearly 
$|d_{k,X}|$ is also the number of illegitimate chains of length $k$, where $V_k$ contains $C_X=\vee_{i \in X}C_i$.
Here the join $\vee$ in the lattice of $q$-cycles of $\M$ is given by the linear span of the vector spaces involved. Unlike the case of classical matroids in Section \ref{classical}, we will no longer separate the $X \subseteq S$ into only two different cases such as $C_X = \mathcal{E}$ or not. Rather, we are dividing into separate cases for each possible codimension for $C_X$ between zero and the maximal codimension of a $q$-circuit of $\M$.

Now we apply Lemma \ref{help2} to prove the statement of Theorem \ref{conj:1} by a straightforward computation.
Assume first that the nullity of the $q$-matroid $\eta(\M) \ge 2$.
For each $X\subset S$ with $C_X=\mathcal{E}$, the contribution to
$\sum\limits_{k=0}^{n} (-1)^k d_k$ from the chains $V_1 \subset\cdots\subset V_k$ with $C_X \subset V_k$, will be that from the chains with $V_k=\mathcal{E}$ and $V_1 \ne \{0\}$. From Lemma \ref{help2}, part (1), this is $(-1)^{n}q^{{n \choose\ 2}}.$

For each $X \subset S$ with $\dim C_X=n-l$, the contribution to $\sum\limits_{k=0}^{n} (-1)^k d_k$  from the chains $V_1 \subset\cdots\subset V_k$ with $C_X \subset V_k$, will be a sum of the contributions of those chains with  $V_k=V$ and $V_1 \ne \{0\}$, taken over all subspaces $C_X \subset V \subset\mathcal{E}$.  If $V$ has codimension $j$, for some $0 \le j \le l$, then the contribution to $\sum\limits_{k=0}^{n} (-1)^k d_k$ from the chains with $V_k=V$ is $(-1)^{n-j}q^{{n-j \choose\ 2}}.$
This follows from Lemma \ref{help2}, part (1), again.
Finally, for each $0 \le j \le l$, the number of
subspaces of $E$ containing $C_X$ of codimension $j$ in $E$ is ${l \brack\ j}_q$. This, in combination with the formula for $\mu(L_{\M})$ from \cite{R}, proves the theorem  since the $\lambda_{1,l},\cdots,\lambda_{s,l},$ for each fixed $l$, play an analogous role with respect to our usage of the inclusion-exclusion principle as 
$(\lambda_{1,0},\cdots,\lambda_{s,0})=(0,\lambda_{2},\cdots,\lambda_{s})$ do, 
for $l=0$. Here $\lambda_i$ is the number of subsets of $S=\{1,2,\cdots,s\}$ containing $i$ elements whose linear span is $E$, using the pattern from \cite[Prop. 1]{R}. (Here $\lambda_{1,0}=\lambda_1=0$, since we have assumed that the nullity $\eta(\M) \ge 2$, but $\lambda_{1,l}$ is not necessarily $0$ for $l \ge 0$. In fact it is equal to the number of $q$-circuits of codimension $l$.

Now we prove the case of $\eta (\M)=1$. In this case, there is only one circuit $C_1$ of the $q$-matroid $\M$ and let $l$ be the codimension of $C_1$ in $E$. If $\M$ is coloop-free, then $C_1 =E$ and thus $l=0$. Thus following the above argument, $\chi(I_{\M})=(-1)^{n}q^{{n \choose\ 2}}$. On the other hand, if $\M$ has coloop, then $C_1 \neq\mathcal{E}$ and thus $l \ge 1$. Thus for $l \ge 1$, we have $$\chi(I_{\M}) =  \sum\limits_{j=0}^l(-1)^{n+j}q^{{n-j \choose\ 2}}{l \brack\ j}_q.$$ Now using the fact that $\overline{\mu}(\M)=1$ or $0$ depending on whether  $\M$ is coloop-free or not, we get that 
$$\chi(I_{\M})= (-1)^{r-1}q^{{n \choose\ 2}}\overline{\mu}(\M)+\sum\limits_{j=0}^l(-1)^{n+j}q^{{n-j \choose\ 2}}{l \brack\ j}_q.$$
 For $\eta(\M)=0$, the lattice of $q$-cycles is empty. Therefore, the result doesn't give much meaning, and $\chi(I_M)=0$ then.)

\end{proof}

\begin{rem} \label{use}
{\rm The term $(-1)^{r-1}q^{{n \choose\ 2}}\overline{\mu}(\M)$ is equal to 
$(-1)^{n-1}q^{{n \choose\ 2}}{\mu}(\M)$ if $\M$ is coloop-free, (and $0$ otherwise). This is because
$(-1)^{r-1}\overline{\mu}(\M)=(-1)^{r-1}(-1)^{n(M)}\mu(L)=(-1)^{(r-1)+(n-r)}\mu(L)=(-1)^{n-1}\mu(L).$ 
In the proof below of Theorem \ref{conj:1} we will
use the form $(-1)^{n-1}q^{{n \choose\ 2}}\mu(\M).$ It should also be noted that if we let the $l$ run from $0$ to $n-1$ instead of $1$ to $n-1$ in the right part of the right side of the equation in Theorem 4.1, then it would “eat” the left part 
$(-1)^{n-1}q^{n \choose\ 2}\overline{\mu}(\M)$, so that we get a formula with a new, revised right part:
$$\chi(\M)=\Sigma_{l=0}^{n-1}\Sigma_{i=1}^s\lambda_{i,l}\Sigma_{j=0}^l(-1)^{n+i+j-1}q^{{n-j \choose\ 2}}{l \brack\ j}_q.$$ 
We have chosen to present it as (\ref{mainform}) in Theorem \ref{conj:1}, though, since we think (\ref{mainform}) illustrates the similarity with, and the difference from the case of usual matroids, in a better way.}
\end{rem}

\begin{rem} \label{extra}
{\rm  Setting $q=1$ in the formula in Theorem \ref{conj:1}, we obtain just $\overline{\mu}(\M)$ up to sign, like in Theorem \ref{main1}, since  $\Sigma_{j=0}^{l}(-1)^j{ l \brack j}_q=0$. Hence Theorem \ref{conj:1} could  be viewed as some sort of a $q$-analog of Theorem \ref{main1}.}
\end{rem}
We also have:

\begin{prop} \label{onlyonehom}
Let $\M$ be a $q$-matroid of rank $r$. Then, using reduced homology over $\mathbb{Z}$, we have:
$$h_i(I_{\M})=0\textrm{ if }i \ne r-1,$$
$$h_{r-1}(I_{\M})=\mathbb{Z}^{|\chi(I_{\M})|},$$
where $|\chi(I_{\M})|$ was given in Theorem \ref{conj:1} and Remark \ref{extra} for $\eta(\M) \ge 1$.

Moreover this is also the reduced singular homology of the 
finite set of non-zero $q$-independent spaces of $\M$, with the order topology.
\end{prop}
\begin{proof}
Since one has proved (see \cite[Theorem 3.9]{P22}) that the chain complex is shellable, as it is obtained from a $q$-shellable complex, all its homology is concentrated in one term. This is for degree equal to the dimension of the simplicial complex, which is $r-1$ here,
and the rank of the top homology group is then equal to the absolute value of its Euler characteristic. 

Furthermore  the homotopy equivalence from Proposition \ref{equival} implies the last statement.
\end{proof}

\begin{example} \label{uniform}
{\rm Look at the uniform $q$-matroid $\M=U(n-2,n)$ with ground space $E=\mathbb{F}^n$.
The circuits are all subspaces of dimension $n-1$,
and the single space of nullity $2$ is $E$.
Hence, in the poset $L$, representing $q$-cycles of $\M$, we have: 
$$\overline{\mu}(\M)=\mu(L)=\mu(0,1)=-\Sigma_{x < 1}\mu(0,x)=(q^{n-1}+q^{n-2}+\cdots+1)-1=$$
\[q^{n-1}+q^{n-2}+\cdots+q,\]
since each of the $x$ representing $q$-cicuits have $\mu(0,x)=-1$, and $\mu(0,0)=1$ by convention. Here $0$ represents the $q$-cycle $\{0\}$ and $1$ represents the $q$-cycle $E$.

The only non-empty $X \subset S$ with $C_X \ne\mathcal{E}$ are the singletons in $S=\{1,2,\cdots,q^{n-1}+q^{n-2}+\cdots+1\}$, and for each of these we have $\dim C_X=n-1$. Hence $\lambda_{1,1}=q^{n-1}+q^{n-2}+\cdots+1$, and all other $\lambda_{i,r}=0$ in the formula of Theorem \ref{conj:1}.
Hence that formula gives
$$\chi(I_{\M})=$$
$$(-1)^{n-1}q^{{n \choose\ 2}}\mu(L)
+(q^{n-1}+q^{n-2}+\cdots+1)((-1)^{n}q^{{n \choose\ 2}}+(-1)^{n-1}q^{{n-1 \choose\ 2}})=$$
$$(-1)^{n}(q^{{n \choose\ 2}}-(q^{n-1}+q^{n-2}+\cdots+1)q^{{n-1 \choose\ 2}}=$$
\begin{equation} \label{tjo}
(-1)^{n-1}(q^{n-2}+\cdots+q+1)q^{{n-1 \choose\ 2}}.
\end{equation}

For $n=2$ we get $\chi(I_{\M})=-{f}_0=-1$, as predicted from (\ref{tjo}),  only counting the empty chain.

For $n=3$ we get $\chi(I_{\M})=-f_0+f_1=(q^2+q+1)-1=(q^2+q)$,
representing $q^2+q+1$ chains of length $1$ (the one-dimensional spaces) and $1$ empty chain. This is also what we get from Equation (\ref{tjo}).

For $n=4$ we get $\chi(I_{\M})=-f_0+f_1-f_2$, where $f_0=1$ represents the empty chain, $f_1=(q^3+q^2+q+1)+(q^4+q^3+2q^2+q+1)$ is the number of  chains of length $1$, so $f_1$ is the sum of the number of one- and two-dimensional spaces in $E=\mathbb{F}_q^4$. 

Furthermore $f_2=(q^3+q^2+q+1)(q^2+q+1)$ represents the incidences (by inclusion) of one- and two- 
dimensional spaces in $E=\mathbb{F}_q^4.$ The number
$-f_0+f_1-f_2$ is then equal to $-(q^5+q^4+q^3),$ which is the same as  we get when inserting $n=4$ into the formula (\ref{tjo}). }

\end{example}

\begin{rem}
{\rm The purpose of Example \ref{uniform} was not to find  a new result, but to demonstrate how Theorem \ref{conj:1} can be used to make calculations in practical examples. In \cite[Theorem 5.11]{GPR21} the more general formula $\chi(I_{\M})=q^{k(k+1)/2}{ n-1 \choose\ k}_q$ is given, for $M=U(k,n)$, for all $0 \le k \le n$, not only $k=n-2$. One easily checks that the formula in \cite{GPR21} gives the same value as the one given in Example \ref{uniform}, if $k=n-2$. The reader is encouraged to verify that Theorem \ref{conj:1} gives results coinciding with that in \cite{GPR21} for other values of $k$ than $n-2$.}
\end{rem}

\begin{example}
{\rm We study an example of a $q$-matroid with both  loops and a coloop. 
Let $P_1$ be the $q$-matroid on $\mathbb{F}_{2^2}$ of rank 1 represented by 

\[\begin{bmatrix} 
 0 & 0 &1 
 \end{bmatrix}_{\mathbb{F}_{2^2}}\]
       \vfill
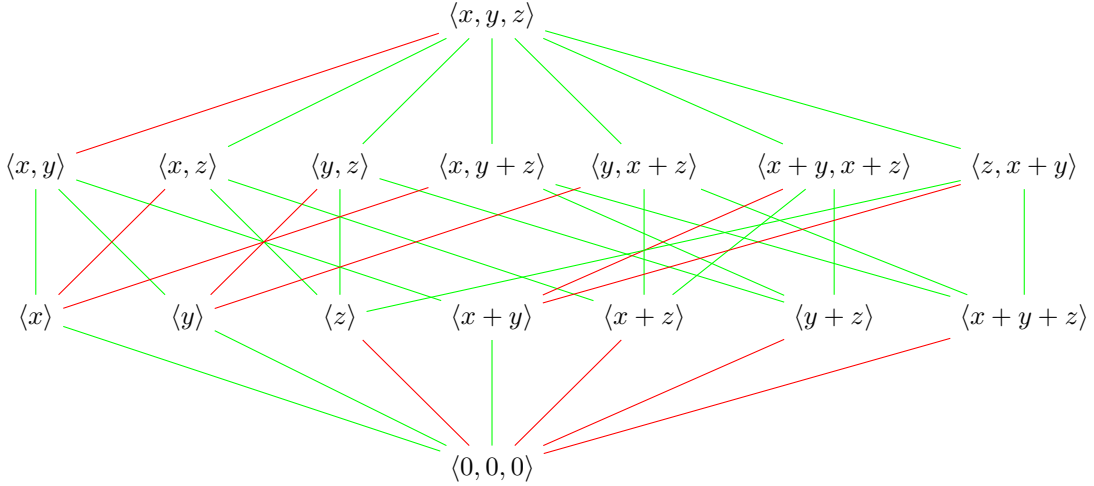
\begin{figure}[H]\label{fig:S2}
\begin{center}
\hspace*{-.3in}
\begin{small}
\begin{tikzpicture}
  \node (max) at (0,2) {$\<x,y,z\>$};
  \node (xy) at (-6,0) {$\<x,y\>$};
  \node (xz) at (-4,0) {$\<x,z\>$};
  \node (yz) at (-2,0) {$\<y,z\>$};
  \node (xyz) at (0,0) {$\<x,y+z\>$};
  \node (yxz) at (2,0) {$\<y,x+z\>$};
  \node (xyxz) at (4.5,0) {$\<x+y,x+z\>$};
  \node (zxy) at (7,0) {$\<z,x+y\>$};
  \node (x) at (-6,-2) {$\<x\>$};
  \node (y) at (-4,-2) {$\<y\>$};
  \node (z) at (-2,-2) {$\<z\>$};
  \node (XY) at (0,-2) {$\<x+y\>$};
  \node (XZ) at (2,-2) {$\<x+z\>$};
  \node (YZ) at (4.5,-2) {$\<y+z\>$};
  \node (XYZ) at (7,-2) {$\<x+y+z\>$};
  \node (min) at (0,-4) {$\<0,0,0\>$};
  \draw [green] (min) -- (x);
  \draw [green] (min) -- (y);
  \draw [red] (min) -- (z);
  \draw [green] (min) -- (XY);
  \draw [red] (min) -- (XZ);
  \draw [red] (min) -- (YZ);
  \draw [red] (min) -- (XYZ);
  \draw [red] (max) -- (xy);
  \draw [green] (max) -- (yz);
  \draw [green] (max) -- (xz);
  \draw [green] (max) -- (xyz);
  \draw [green] (max) -- (yxz);
  \draw [green](max) -- (zxy);
  \draw [green] (max) -- (xyxz); 
  \draw [green] (xy) -- (x);
  \draw [green] (xy)-- (y);
  \draw [green] (xy)-- (XY); 
  \draw [red] (xz) -- (x);
  \draw [green] (xz) -- (z);
  \draw [green] (xz) -- (XZ);
  \draw [red] (yz) -- (y);
  \draw [green] (yz) -- (z);
  \draw [green] (yz) -- (YZ);
  \draw [red] (xyz)-- (x);
  \draw [green] (xyz)-- (YZ);
  \draw [green] (xyz)-- (XYZ);
  \draw [red] (yxz)-- (y);
  \draw [green] (yxz)-- (XZ);
  \draw [green] (yxz)-- (XYZ);
  \draw [green] (xyxz)-- (XZ);
  \draw [red] (xyxz)-- (XY); 
  \draw [green] (xyxz)-- (YZ);
  \draw [green] (zxy) -- (z);
  \draw [red] (zxy) -- (XY);
  \draw [green] (zxy) -- (XYZ);
\end{tikzpicture}
\end{small}
\end{center}
\caption{Illustration of $P_1$ by bi-coloring}
\end{figure}
Here a red line segment means that the rank stays constant, and a green one means that the rank increases upwards in the diagram.
The example is case A.4.4 of \cite{CJ}. Here $\overline{\mu}=0$ because of the coloop; the top $q$-cycle, of nullity $2$, is not $E=\mathbb{F}_2^3$, but the subspace given by $Z=0$. Moreover one sees that $\lambda_{1,2}=3$ since all $3$ circuits have codimension $2$. Furthermore 
$\lambda_{2,1}=3$, since all $3$ choices of two circuits span a codimension one space. Moreover $\lambda_{3,1}=1$,
since the single choice of $3$ circuits (all of them) span the top $q$-cycle, of codimension $1$.
The  formula of Theorem \ref{conj:1} then gives:

\begin{align*}
\chi(I_{\M})&= 0 + q_{1,2}(-2^{3 \choose 2}+
2^{2 \choose 2}-2^{1 \choose 2})+
\lambda_{2,1}(2^{3 \choose 2}-2^{2 \choose 2})-\lambda_{3,1}(2^{3 \choose 2}-2^{2 \choose 2})\\
&=3(-8+6-1)+3(8-2)-1(8-2)=3.
\end{align*}
Direct inspection shows that the only independent spaces are $\{0\}$ and $4$ one-dimensional spaces.
That gives the empty chain and $4$ chains of length $1$ as the only chains of $I_{\M}$, so hence we indeed have $\chi(I_{\M}) =-1+4=3$.} 

\end{example}

\section{A characterization of nonzero Euler characteristic of the order complex}\label{sec:5}

In this section we show that, for $q$-matroids $\M$, the Euler characteristic $\chi(I_{\M})$ of $I_{\M}$ always is non-zero, except in one special case. This is in sharp contrast to the case for usual matroids $M$, where $\chi(I_{M})=0$ if and only if $M$ has a coloop, i.e., if and only if $E$ is not a cycle. The results of this section are mainly consequences of the shellability property of both $q$-matroid complexes $\M$ and their associated order complexes $I_{\M}$.

For convenience, we recall some notations from Section \ref{2.2}. We fix an arbitrary total order $\prec$ on $\Fq$ such that $0 \prec 1 \prec \alpha$  for all $\alpha \in \Fq\backslash \{0,1\}$. This extends lexicographically to a total order on $\E=\Fq^n$, 
which we also denote by~$\prec$. 
We use $\prec_q^i$ to denote the total order on the Grassmannian $G(i,n)$ for $1 \leq i \leq n$. For a $q$-matroid complex $\Delta_{\M}$ of dimension $k$, we simply use $\prec_q$, to denote the total order induced by $\prec_q^k$ on its facets as given in Definition \ref{prec}. We denote by $\prec_{\ell}$ the reverse lexicographic order on the facets (or maximal chains) of the order complex $I_{\M}$ induced by orderings $\prec_q^i$.   Recall that the total order $\prec_q$ is a shelling on $\Delta_{\M}$ and so is the total order $\prec_{\ell}$ on the order complex $I_{\M}$. 

In what follows, for any nonempty subset $S$ of $\E$, 
we denote by $\min S$ the least element of $S$ with respect to the total order $\prec$ on $\E$ defined above. For a $q$-complex $\Delta$, we use $\min \Delta$ to denote the minimum nonzero element $x$ with respect to the total order $\prec$ on $\E$ such that $\langle x \rangle \in \Delta$.

\begin{defn}\label{def:lead}\cite[Definition 3.1]{GPR21}
For any nonzero vector $u\in \Fq^n$, the \emph{leading index} of $u$, denoted as $p(u)$, is defined to be the least positive integer $i$ such that the $i$-th entry of $u$ is nonzero. Moreover, the \emph{profile} $p(S)$ of a subset $S$ of $\Fq^n$ is defined to be the set of the leading indices of all of its nonzero elements, i.e., 
\[
p(S) =  \{p(u)\colon u\in  S\backslash\{\0\} \}.
\]
Note that $p(S)$ can be the empty set if $S$ contains no nonzero vector.
\end{defn}

\begin{lemma}
Given a subspace $U$ of $\Fq^n$, let $[u_k, \ldots, u_1]^T \in \Fq^{k \times n}$ be the reduced row echelon form of $U$. Then for any nonzero element $u$ of $\Fq^n$, $p(u)$ is equal to $p(u_i)$ for some $i = 1, \ldots, k$. 
\end{lemma}

\begin{proof}
 It is clear that if $u = c_1 u_1 + \cdots + c_i u_i$ with $c_i \in \Fq$ and $c_i \neq 0$, then $p(u) = p(u_i)$. Hence the lemma is proved.
\end{proof}

\begin{lemma}\label{lemma:lastrow}
If $u = \min_{\prec}U \backslash \{0\}$ for a subspace $U \subseteq \Fq^n$, then $u$ will appear as the last row of reduced row echelon form of $U$. 
\end{lemma}

\begin{proof}
 Let $[u_k, u_{k-1}, \ldots, u_1]^T \in \Fq^{k \times n}$ be the reduced row echelon form of $U$. Then we know that $p(u_1) > p(u_2) > \cdots, p(u_k)$. Since $u = \min_{\prec}U \backslash \{0\}$, we have $u \prec u_1$ and thus $p(u) \geq p(u_1) > p(u_2) > \ldots > p(u_k)$. Since $u \in U$, $p(u)$ must be equal to some $p(u_i)$ for $i= 1,\ldots, k$, which implies $p(u) = p(u_1)$. But if $u \neq u_1$, then $p(u-u_1) > p(u_1) > \ldots > p(u_k)$ which is a contradiction to the fact that $p(u - u_1)$ is not equal to any $p(u_i)$ for $i=1, \ldots, k$.
\end{proof}

\begin{lemma} \label{nocommonx}
Let $\M$ be a nontrivial $q$-matroid on $\Fq^n$ of rank $k$ such that $0< k < n$. Then there is no nonzero vector $x \in \Fq^n$ such that it belongs to all the basis elements of $\M$.
\end{lemma}

\begin{proof}
 Let $\mathcal{B}_{\M}$ be the set of bases of $\M$. Now assume that $x \in \Fq^n\setminus \{0\}$ is such that $x \in B$ for all $B \in \mathcal{B}_{\M}$. We show that with this assumption $\mathcal{B}_{\M}$ will not satisfy the following axiom for bases of a $q$-matroid: \textit{(B4) Let $A, B \subseteq \Fq^n$ and $I, J$ be maximal intersections of some elements of $\mathcal{B}_{\M}$ with $A$ and $B$, respectively. Then there exists a maximal intersection of some element in $\mathcal{B}_{\M}$ with $A+B$ which is contained in $I+J$.}
 
 Now let $F_1 \in \mathcal{B}_{\M}$ be such that $F_1 = \< x, x_2,\ldots, x_k\>_{\Fq}$, i.e., $\{x,x_2, \ldots, x_k\}$ form a basis of $F_1$ and we extend this to a basis $\{x, x_2, \ldots, x_k, x_{k+1},\dots, x_n\}$ of $\Fq^n$. Consider the following two subspaces:
 \begin{align*}
     A&=\<x_2, \ldots, x_k, x_{k+1}\>,\\
 B&=\<x_2, \ldots, x_k, x_{k+1}+x\>.
 \end{align*}
Note that $x \notin A$ because $\{x, x_2, \ldots, x_{k+1}\}$ is a linearly independent set. Also, $x \notin B$ since otherwise, $\{x,x_2,\ldots,x_{k+1}\} \subseteq B$ which is impossible by comparing their dimensions. Since $x$ is not in both $A$ and $B$, $I:= A\cap F_1$ and
$J:=B \cap F_1$ are maximal intersections of some elements, in this case $F_1$, of $\mathcal{B}_{\M}$ with $A$ and $B$, respectively. Now $I = J = \{x_2, \ldots, x_k\}$ and thus $I+J = I$, which does not contain $x$. But $A+B$ contains $x$ and so does all the basis elements in $\mathcal{B}_{\M}$. Therefore, any maximal intersection of some element in $\mathcal{B}_{\M}$ with $A+B$ cannot be contained in $I+J$. This proves that our assumption is wrong. So there cannot be any nontrivial $q$-matroid such that all the basis elements contain a common nonzero vector in $\Fq^n$.

\end{proof}

\begin{prop}\label{isth1}
Let $\M$ be a $q$-matroid on $\mathcal{E}$ of rank $r$. If the bases of $\M$ do not share any common nonzero element among them, then $\chi(I_{\M}) \neq 0$.
\end{prop}

The main idea to prove the above statement is to show that when the bases of $\M$ do not share any common nonzero element among them, then there is a maximal chain $A \in I_{\M}$ such that $\R(A) = A$. First we recall some results from \cite{P22} and prove a short lemma that will be used in the proof of the proposition.

 \begin{prop}\cite[Proposition 4.4]{P22}\label{PRO:4.4}
 Let $\Delta$ be a shellable $q$-complex on $\mathcal{E}$ of dimension $r$ and $K(\Delta)$ be the corresponding order complex. Let $F_1, \ldots, F_t$ be the shelling $\prec_q$ on $\Delta$ and $\prec_\ell$ be the shelling on the order complex $K(\Delta)$. Then for a maximal chain $A = \{A_0 < A_1 \cdots < A_r=F_j\}$ in $K(\Delta)$, $\R(A) = A$ if and only if there exists $1 \leq i <j\leq t$ such that $A_{r-1} \subseteq F_i$ and $A_k \neq \min_{\prec_q^k} \{ F : A_{k-1} < F < A_{k+1}\}$ for $1 \leq k < r$. 
 \end{prop}
 
 \begin{cor}\cite[Corollary 4.6]{P22}\label{cor:4.5}
 Let $A = \{A_0 < A_1 \cdots < A_r\}$ be a maximal chain in the order complex $K(\Delta)$ of a $q$-complex $\Delta$ of dimension $r$. Then $A_k \neq \min_{\prec_q^k} \{ F : A_{k-1} < F < A_{k+1}\}$ for $1 \leq k < r$ if and only if $A_{k}$ does not contain the minimum nonzero vector of $A_{k+1}$ for all $1 \leq k < r$.
 \end{cor}
 \begin{lemma}\label{intermediate}
Let $F_1, \ldots, F_t$ be the facets of a $q$-matroid $\M$ on $E = \Fq^n$ of rank $r$ and let $I_{\M}$ be the corresponding order complex. Then there exists a maximal chain $A$ in $I_{\M}$ with $\R(A) = A$ if and only if for some facet $F_j$, there exists $F_i$ with $1 \leq i < j \leq t$ such that $\dim_{\Fq} F_i \cap F_j = r-1$ and $\min F_j \backslash 0 \notin F_i$.
\end{lemma}

\begin{proof}
If $\R(A) = A$, then the existence of such $F_i$ follows immediately from Proposition \ref{PRO:4.4}.

Conversely, we assume that for some facet $F_j$ of $\M$ there is a facet $F_i$ with $1 \leq i < j \leq t$ such that $\dim_{\Fq} F_i \cap F_j = r-1$ and $\min F_j \backslash 0 \notin F_i$. Set $A_{r-1} = F_i \cap F_j$. Now, for any $s$ with $1 \leq s \leq r-2$, we inductively choose $A_s$ to be a codimension $1$ subspace of $A_{s+1}$ so that $A_s$ does not contain the minimum nonzero element of $A_{s+1}$. Then combining Proposition \ref{PRO:4.4} and Corollary  \ref{cor:4.5}, we have $\R(A) = A$.
\end{proof}

\begin{lemma}
Let $\Delta_{\M}$ be a $q$-matroid complex corresponding to the $q$-matroid $\M = (\Fq^n, \rho)$ and let $F_1, \ldots, F_t$ be a shelling on the facets of $\Delta_{\M}$. If $\phi$ is an $\Fq$-linear isomorphism of $\Fq^n$, then the set $\{\phi(F_1), \ldots, \phi(F_t)\}$ forms a basis of a $q$-matroid, which we call as $\phi(\M)$. Moreover, $\phi(F_1), \ldots, \phi(F_t)$ is a shelling on the facets of $\Delta_{\phi(\M)}$.
\end{lemma}

\begin{proof}
 Note that any $\Fq$-linear isomorphism preserves dimension and intersection of $\Fq$-spaces. Thus $\phi(F_i)$'s have same dimension as the dimension of $F_i$'s and they satisfy the basis axioms of a $q$-matroid. It proves the first part.
 
For the second statement, note that the shellability property follows only on the dimension of spaces and the inclusion of subspaces. And both these properties are preserved under $\Fq$-linear isomorphism. Thus the isomorphic copies of the facets $F_i$'s when arranged in the order of $F_i$'s will also give a shelling.
\end{proof}

We remark that the shelling $\phi(F_1), \ldots, \phi(F_t)$ on the facets of $\Delta_{\phi(\M)}$ may not be the shelling $\prec_q$ mentioned in Theorem \ref{q-shell}. 
\begin{lemma}
Two isomorphic $q$-complexes have the same homology groups when equipped with order topology with the partial order given by subspace inclusion.
\end{lemma}

\begin{proof}
It is clear from the fact that the underlying posets of two isomorphic $q$-complexes are same and the order topology (or the homology groups) depends only on the poset structure.
 \end{proof}
\begin{proof}[Proof of Proposition \ref{isth1}]
 Because of the shellability property of the chain complex $I_{\M}$, it is enough to prove that the only nontrivial (singular) homology group of the (punctured) $q$-matroid complex $\mathring{\Delta}_\M$ has a positive rank. Equivalently, it is enough to show that there is a maximal chain $A$ in the shellable simplicial complex $I_{\M}$ such that $\R(A) = A$. We consider the characterization of such maximal chains as given in \cite[Proposition 4.4]{P22} to show that there exists at least one maximal chain $A$ with $\R(A) = A$ given the condition that there is no common nonzero element in all the bases of $\M$. 
 
 Let $F_1, \ldots, F_t$ be the facets of $\Delta_\M$ ordered according to the shelling induced by tower decomposition $\prec_q$ and 
 let $x = \min \Delta_\M$, i.e. $x$ is the minimum nonzero element w.r.t. the order $\prec$ on $\mathcal{E}$ such that $\left< x \right> \in \Delta_{\M}$.

 We know that the facets containing the minimum nonzero vector in $\Delta_\M$ belong to the beginning part of the shelling order. Since there is no common nonzero vector in all of its facets, there is a positive integer $s < t$ such that $F_1, \ldots, F_s$ contain $x (= \min \Delta_M)$, but $F_{s+1}$ does not. Now to get a legitimate maximal chain $A$ with the property $\R(A)=A$, it suffices to show that there is codimension 1 subspace $A_{r-1}$ of $F_{s+1}$, so that $A_{r-1} \subseteq F_i$ for some $i < s+1$ and $A_{r-1}$ does not contain the minimum nonzero vector of $F_{s+1}$. So we take a facet $F_i$ with $i< s+1$, such that $\dim F_i \cap F_{s+1} = r-1$. Note that shellability of $\Delta_\M$ guarantees the existence of such $F_i$. Let $y=\min  F_{s+1} \backslash F_i$. If $y$ is the minimum nonzero element of $F_{s+1}$, then we are done following Lemma \ref{intermediate}. Otherwise, take $\alpha = \min_{\prec}F_{s+1}\backslash \{0\}$. We define an isomorphism on the elements of $\Fq^n$.

First we choose a suitable basis of $\Fq^n$. Note that the first nonzero entries of $\alpha$ and $y$ are 1. Indeed, since both $\alpha$ and $y$ are minimum nonzero elements with respect to the order $\prec$ in some sets. Thus $\alpha$ and $y$ are linearly independent elements. Since $\dim ~ F_i \cap F_{s+1} = r-1$, then we can consider a basis of the form $\mathcal{B} = \{\alpha, \alpha_1, \ldots, \alpha_{r-2}\}$ of $F_i \cap F_{s+1}$. Note that $y \in F_{s+1}\backslash F_i$ implies that $\mathcal{B} \cup \{y\}$ is a basis of $F_{s+1}$. Moreover, $\mathcal{B} \cup \{x\}$ is a basis of $F_i$. Now we extend the set $\mathcal{B} \cup \{y, x\} = \{\alpha, \alpha_1, \ldots, \alpha_{r-2}, y, x\}$ to a basis $\tilde{\mathcal{B}}$ of $\Fq^n$. Consider the isomorphism $\phi \colon \Fq^n \longrightarrow \Fq^n$, that maps $\alpha$ to $y$ and $y$ to $\alpha$ and keeps other basis elements in $\tilde{\mathcal{B}}$ fixed. Now note that $\dim \phi(F_i) \cap \phi(F_{s+1})$ is unchanged under the isomorphism of $\Fq^n$ and $\phi(F_{s+1}) = F_{s+1} $. 
 As the isomorphism $\phi$ preserves dimensions of subspaces and their intersection, it is clear that $\phi(F_1), \ldots, \phi(F_t)$ are the facets of the isomorphic $q$-matroid complex, which we call as $\phi(\mathcal{M})$.
 Also, note that $\phi(x)=x$ and this implies that $x \in \phi(F_i)$. Moreover, $\{x, y, \alpha_1, \ldots, \alpha_{r-2}  \}$ is a basis of $\phi(F_i)$. 
 
We already know that $\prec_q$ is a shelling on the facets of any $q$-matroid complex (Theorem \ref{q-shell}). Here we claim that, for the $q$-matroid complex  $\phi(\mathcal{M})$, the facet $\phi(F_i)$ appears before the facet $F_{s+1}$ with respect to the shelling $\prec_q$. 
 
 \textbf{Proof of the claim: } 
 Let $x' = \min \phi(F_i)\backslash \{0\}$. Since $x \in \phi(F_i)$, $x' \prec x$. On the other hand, $x = \min \Delta_{\M}$ implies $x \prec \alpha$, where $\alpha = \min_{\prec} \phi(F_{s+1}) \backslash \{0\}$. 
 Now following Lemma \ref{lemma:lastrow}, $x'$ and $\alpha$ will appear as the last rows in the r.r.e.f. of 
 $\phi(F_i)$ and $F_{s+1}$. Hence $x' \prec \alpha$ implies that $\phi(F_i) \prec_q \phi(F_{s+1}) = F_{s+1}$. This proves our claim. 
 
 Thus if we choose $\phi(A_{r-1}) = \phi(F_i) \cap \phi(F_{s+1})$, then $\alpha \notin \phi(A_{r-1})$ since $\alpha = \phi(y) $ cannot be in $\phi(F_i)$. Now following Lemma \ref{intermediate}, we get that there exists a maximal chain $A$ ending with $\phi(A_{r-1}) \subseteq \phi(F_{s+1}$) such that $\R(A)=A$. Hence it proves that the $\chi(I_{\phi(\M}))$ is nonzero and so is $\chi(I_{\M})$. Therefore, the Euler characteristic of the chain complex of a $q$-matroid complex is indeed nonzero if there is no common non-zero vector $x$ contained in all the bases of the $q$-matroid.
\end{proof}
Combining Lemma \ref{nocommonx} and Proposition \ref{isth1}, we end this section by concluding:

\begin{cor}
If a $q$-matroid is of type $U(n,n)$, then 
$\chi(I_{\M})=0$. For all other $q$-matroids we have $\chi(I_{\M}) \ne 0$.
\end{cor}
\begin{proof}
The statement for $\M=U(n,n)$ follows from \cite[Theorem 5.11]{GPR21}  (or from Formula (\ref{Dexpr}), since $d_k=0$ for all $k$). The statement for all other $q$-matroids follows from Lemma \ref{nocommonx} and Proposition \ref{isth1}.
\end{proof}

We end this section by showing an example of a situation which is impossible for the analogue for usual matroids: A  one-dimensional subspace of $E$ is not contained in the linear span of the $q$-cirquits, and there exists a basis, in which it is not contained. 

\begin{example}

Let $P^*_1$ be the $q$-matroid on $\mathbb{F}_{2^2}$ of rank 2 represented by 

\[\begin{bmatrix} 
1 & 0 & 0 \\
0 &\alpha &1
       \end{bmatrix}_{\mathbb{F}_{2^2}}\]
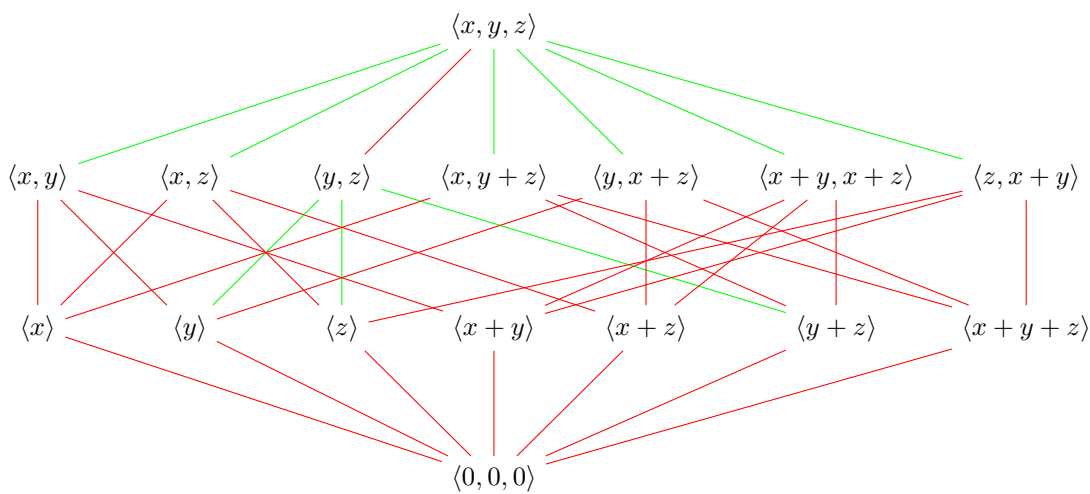
\begin{figure}[H]\label{fig:S3}
\begin{center}
\hspace*{-.3in}
\begin{small}
\begin{tikzpicture}
  \node (max) at (0,2) {$\<x,y,z\>$};
  \node (xy) at (-6,0) {$\<x,y\>$};
  \node (xz) at (-4,0) {$\<x,z\>$};
  \node (yz) at (-2,0) {$\<y,z\>$};
  \node (xyz) at (0,0) {$\<x,y+z\>$};
  \node (yxz) at (2,0) {$\<y,x+z\>$};
  \node (xyxz) at (4.5,0) {$\<x+y,x+z\>$};
  \node (zxy) at (7,0) {$\<z,x+y\>$};
  \node (x) at (-6,-2) {$\<x\>$};
  \node (y) at (-4,-2) {$\<y\>$};
  \node (z) at (-2,-2) {$\<z\>$};
  \node (XY) at (0,-2) {$\<x+y\>$};
  \node (XZ) at (2,-2) {$\<x+z\>$};
  \node (YZ) at (4.5,-2) {$\<y+z\>$};
  \node (XYZ) at (7,-2) {$\<x+y+z\>$};
  \node (min) at (0,-4) {$\<0,0,0\>$};
  \draw [red] (min) -- (x);
  \draw [red] (min) -- (y);
  \draw [red] (min) -- (z);
  \draw [red] (min) -- (XY);
  \draw [red] (min) -- (XZ);
  \draw [red] (min) -- (YZ);
  \draw [red] (min) -- (XYZ);
  \draw [green] (max) -- (xy);
  \draw [red] (max) -- (yz);
  \draw [green] (max) -- (xz);
  \draw [green] (max) -- (xyz);
  \draw [green] (max) -- (yxz);
  \draw [green](max) -- (zxy);
  \draw [green] (max) -- (xyxz); 
  \draw [red] (xy) -- (x);
  \draw [red] (xy)-- (y);
  \draw [red] (xy)-- (XY); 
  \draw [red] (xz) -- (x);
  \draw [red] (xz) -- (z);
  \draw [red] (xz) -- (XZ);
  \draw [green] (yz) -- (y);
  \draw [green] (yz) -- (z);
  \draw [green] (yz) -- (YZ);
  \draw [red] (xyz)-- (x);
  \draw [red] (xyz)-- (YZ);
  \draw [red] (xyz)-- (XYZ);
  \draw [red] (yxz)-- (y);
  \draw [red] (yxz)-- (XZ);
  \draw [red] (yxz)-- (XYZ);
  \draw [red] (xyxz)-- (XZ);
  \draw [red] (xyxz)-- (XY); 
  \draw [red] (xyxz)-- (YZ);
  \draw [red] (zxy) -- (z);
  \draw [red] (zxy) -- (XY);
  \draw [red] (zxy) -- (XYZ);
\end{tikzpicture}
\end{small}
\end{center}
\caption{Illustration of $P^*_1$ by bi-coloring}
\end{figure}

For this $q$-matroid, $<x>$ is a one-dimensional subspace of the ground space, which is not contained in the linear span of the $q$-circuits, since the only $q$-circuit of $P^*_1$ is $\langle y,z\rangle$. Note that $x$ is not contained in all the bases of $P^*_1$. Thus this example contradicts the statement that if a one-dimensional subspace of the ground set of a $q$-matroid is not contained in the linear span of the $q$-circuits, then that subspace belongs to all the basis of the $q$-matroid. The analogue of this statement, as well as its converse, is true for classical matroids.
\end{example}

\section{An attempt to describe generalized rank weights in terms of singular homology}\label{sec:6}

First, we recall the notion of support for Gabidulin rank metric codes. We fix a basis $\{a_1,\dots,a_m\}$ of the field extension $\Fqm/\Fq$. If $x=\sum_{i=1}^m \lambda_i a_i \in \Fqm$, we define the column vector $\phi(x)=(\lambda_1,\dots,\lambda_m)^T$.
\begin{defn}
For a vector $\xx=(x_1,\dots,x_n)\in \Fqm^n$, the support of $\xx$ is defined as the $\Fq$-rowspace $Supp(\xx)=Rowsp(\phi(x_1),\dots,\phi(x_n))$. More generally, if $D$ is a subspace of $\Fqm^n$, the support $supp(D)$ is the $\Fq$-subspace defined by the linear span of all the supports of the vectors in $D$.
\end{defn}

Throughout this section, whenever we talk about subcode, it is meant to be a $\Fqm$-linear subspace of $C$. Let $C$ be a Gabidulin rank-metric code where we for convenience set  $\M=\M_C^*$. Then we have the following well-known result, of which we present a proof for the benefit of the reader:
\begin{lemma} \label{support}
The supports of $\Fqm$-linear subcodes of $C$ are precisely the $q$-cycles of $\M$.
\end{lemma}
\begin{proof}
Let $U$ be an $\Fq$-subspace of $E$ which is a $q$-cycle with nullity $s$. By definition, there is a subcode $C_s$ of $C$ of dimension $s$, but no subcode of dimension $s+1$, with support contained in $U$. That $U$ is a $q$-cycle of nullity  $s$, i.e. a minimal subspace of 
nullity $s$, means that the support of $C_s$ not is
strictly contained in $U$. Thus $U$ is support of a subcode and hence $q$-cycles are supports of subcodes. 

Conversely, we can assume that $U$ is the support of some subcode $D'$. Let $D$ be the subcode
of $C$ of all codewords supported on $U$, and that implies $D' \subseteq D$. The nullity of $U$ is then $s=\dim D$. That $U$ is not a $q$-cycle then means that $D$ has a strictly smaller subspace $ V \subset U$ as its support. If $D$ indeed had that, then the support of $D'$ would be contained in $V$. But the support of $D'$ was $U$, so that is impossible. Hence $U$ is a $q$-cycle of nullity $s$. Hence supports of subcodes are $q$-cycles.

\end{proof}

For the rank weights $d_r$ of a Gabidulin rank-metric code $C$ we have, by definition:

\begin{defn}
$$d_r=\min \{\dim_{\Fq}\textrm{T | T \text{ is the support of $\Fqm$-linear subcode of C of dimension }} r\}.$$
\end{defn}
From Lemma \ref{support} and its proof we then see:
$$d_r=\min \{\dim\textrm{T | T is a q-cycle of }  \M \textrm{ of nullity } r\}.$$

\begin{rem}
So far all the statements in this section are completely analogous to the corresponding results for usual Hamming codes and their associated matroids.
To put the results in the preceding sections in context, we now sum up some results that are valid for Hamming metric codes $C$, and later we will compare with Gabidulin rank metric codes. 
\end{rem}
We let  $M$ be the matroid corresponding to parity check matrices of $C$.
For this, and any matroid $M$, let $S_M$ be the simplicial complex of the independent sets of $M$, and let $I_M=\K(\mathring{S}_M)$ be the 
simplicial complex whose simplices are chains of non-empty simplices of $S_M$ in addition to the empty chain. We then have:
\begin{prop} \label{qualitative}
For $X \subset E$,  and rank $r(X)=r$, then $X$ is a cycle of $M$ with nullity function $n$ if and only if
\begin{itemize}
\item $\overline{\mu}(M|_{X})=(-1)^{n(X)}\mu(X) \ne 0$ 
\item $\chi(S_{M|_{X}}) \ne 0$
\item $\chi(I_{M|_{X}}) \ne 0$
\item The simplicial homology   $h_i(S_{M|_{X}}) \ne 0$ for some $i$
(and the single $i$ for which it is non-zero, is $r(X)-1$).
\item The simplicial homology   $h_i(I_{M|_{X}}) \ne 0$ for some  $i$ (and the single $i$ for which it is non-zero, is $r(X)-1$).
\item $\beta_{j,X} \ne 0$ for some $j$ (and the unique such $j$ will be $n(X)$), for the $\mathbb{N}^n$-graded  Betti-numbers of the Stanley-Reisner ring of $S_M$ (where $n=|E|)$.
\item $M|_{X}$ has a coloop, in the sense of containing a point $x$ where its dual rank function
is zero.
\item $M|_{X}$ has an isthmus in the sense of a point $x$ contained in all its bases.
\item the restricted code $C_X$ is non-degenerate.
\end{itemize}
\end{prop}
\begin{proof}
That $\overline{\mu}(M|_{X})=(-1)^{n(X)}\mu(X)$ follows from \cite[Proposition 3.10.1]{Stanley}. The fact that this is non-zero if and only if $M$ has a co-loop (or an isthmus) follows from the definition in Remark \ref{rem:3.5}. The equivalence with the  $\chi(S_{M|_{X}}) \ne 0$ and
$\chi(I_{M|_{X}}) \ne 0$ follows for example from \cite[Formula (7.2)]{B} and Proposition \ref{equival2}. 

For the definition of Stanley-Reisner rings of simplicial complexes and their relations to codes and matroids, and the statements of the proposition, we refer to \cite{JV13}. A key ingredient is Hochster's formula (\cite{H}), which identifies $h_{r(X)-1}(S_{M|_{X}})$ with the relevant Betti number. See the proof of \cite[Theorem 1]{JV13}.
\end{proof}
The knowledge of only the zero-/non-zero-ness of any of these numbers can then be used to determine whether $X$ is a cycle or not, i.e. whether $X$ is the support 
of a subcode, and from the analogue of Lemma \ref{support} and its proof also to determine the generalized Hamming weights $d_r$ of $C$. We also have: 
\begin{rem}
{\rm The first point in the previous list, 
$\overline{\mu}(M|_{X}) \ne 0$, gives meaning, also
when one   \underline{defines}  $\overline{\mu}(M)$ as 
$|-\lambda_1+\lambda_2+\cdots+(-1)^s\lambda_s|$ in the sense of the proof of Theorem \ref{main1}.
With this or the usual definition we have in fact 
\begin{equation} \label{quantify}
\overline{\mu}(M|_{X})=
|\chi(S_{M|_{X}})=|\chi(I_{M|_{X}})|=h_i(I_{S_{M|_{X}}})= h_i(I_{M_{X}})=\beta_{j,X}
\end{equation}
for the relevant $i,j$. By $h_i$ we here mean the homology rank over $\mathbb{Z}$ or over any field.}
\end{rem}
From the results in \cite{JRV} it is clear that the 
Hamming weight distribution, and also all the higher weight spectra of $C$ can be found from the knowledge of all the Betti numbers $\beta_{j,X}$ and corresponding Betti numbers of elongation matroids of $M$ (the cycles of all these elongation matroids are also cycles of $M$, but with lower nullity).
Hence the "quantitive" formula (\ref{quantify}), and not only the "qualitative" (concerning "zeroness" or not) Proposition \ref{qualitative} is in principle useful to determine more than just the generalized Hamming weights.

A main motivation for this article is to investigate how many of the statements above that can be "transferred" to Gabidulin rank-metric codes and the duals $\M$ of their associated codes $\M_C$. It is clear already from the outset that there is no analogous $S_{\M}$ since the independent sets 
of a $q$-matroid do not form a simplicial complex.
But the analoguous $I_{\M}$ does indeed form such a
complex and statements of the kind
${\mu}(\M|_{X})= |\chi(I_{\M_{X}})| =h_i(I_{\M_{X}})$ for the relevant $i$,  for subspaces $X \subset E$ were a priori possible. The same applies to statements about $X$ not being a $q$-cycle equivalent to zeroness of these numbers, again being
equivalent to the existence of co-loops and
isthmuses of $\M|_{X}$. For an isthmus of a $q$-matroid we follow the definition in \cite[Definition 51]{JP} and since a coloop of a $q$-matroid is a loop of the dual $q$-matroid, the equivalence between isthmuses and coloops follows from \cite[Corollary 52]{JP}.

The results in the previous sections are negative, in that they disprove these statements. 
First of all Theorem \ref{conj:1} shows that  $\overline{\mu}(\M|_{X})=|\chi(I_{\M|_{X}})|$
is in general false. It is true that
$|\chi(I_{\M_{X}})| =h_i(I_{M|_{X}})$, for one single $i$ (because of the shellability of $I_{\M|{X}}$), and it is indeed (trivially) true that $X$ is a $q$-cycle if and only if $\M|_{X}$ does not have a coloop. But having a coloop is not equivalent to having a ($1$-dimensional) subspace $x$ contained in all basises of  
$\M|_{X}$. Having a coloop is instead equivalent to 
having a subspace $x$ not contained in $B^{\perp}$ for any basis of the dual matroid of 
$\M|_{X}$. For usual matroids not being contained in the complement of any basis for the dual matroid, is equivalent to being contained all basises of the original one. For $q$-matroids this is not equivalent then. This is a basic reason that Lemma \ref{nocommonx} and
Proposition \ref{isth1} give non-zero 
$\chi(I_{\M|_{X}})$  
not only for cycles $X$, but for all $X$ such that 
$\M|_{X}$ is not uniform of type $U_{t,t}$ for some $t$. Hence the "non-zeroness" of the $\chi(I_{\M|_{X}})$, as opposed to that of the analogous numbers for Hamming codes, are useless in the process of identifying $q$-cycles of $\M$, and thereby supports of subcodes of $C$. 

As for the "quantitative" aspects (how to find the higher weight spectra, once one has identified the $q$-cycles) it is true that the $\overline{\mu}(\M|_X)$, and corresponding numbers for the elongation matroids of $\M$, do determine the higher weight spectra. This was proven in \cite{JPV21}, where these M{\"o}bius numbers were proven to be equal to corresponding numbers defined  via the dual of an associated "classical" matroid, whose M{\"o}bius numbers again were equal to relevant Betti numbers. But there is no reason why  the weight spectra are given in the same way by the 
$\chi(I_{\M|_{X}})$ (and analogous numbers for elongation matrices). Theorem \ref{conj:1}, again,
strongly indicates this.

\vspace{.5cm}
\end{document}